\documentclass[11pt]{article}
\usepackage{color} 
\usepackage{ulem}
\usepackage{subfigure}

\newcommand\redout{\bgroup\markoverwith
{\textcolor{red}{\rule[.5ex]{4pt}{1.4pt}}}\ULon}
\usepackage{xcolor}

\usepackage{amsmath}
\usepackage{amssymb}
\usepackage{amsthm}
\usepackage{graphicx}
\usepackage{bbm} 

\newtheorem{theorem}{Theorem}
\newtheorem{corollary}{Corollary}
\newtheorem{lemma}{Lemma}
\newtheorem{proposition}{Proposition}

\usepackage{ulem}

\usepackage{color}

\def\Var{\text{Var}}
\def\Cov{\text{Cov}}
\def\WZ{\mathcal{Z}}
\def\WG{\mathcal{G}}
\newcommand{\WZtime}[1]{\mathcal{Z}_{#1}}
\newcommand{\WGtime}[1]{\mathcal{G}_{#1}}

\newlength{\noteWidth}
\long\def\notes#1{\ifinner
             {\tiny #1}
             \else
             \marginpar{\parbox[t]{\noteWidth}{\raggedright\tiny #1}}
             \fi}

\newcounter{rmnum}

\def\limsup{\mathop{\rm lim\ sup}}
\def\liminf{\mathop{\rm lim\ inf}}

\def\E{ {\mathbb{E} }}

\def\P{ {\mathbb{P} }}

\usepackage{color}

\def\Zp{{Z^+}}

\def\Var{\text{Var}}
\def\Cov{\text{Cov}}

\def\EN{h}
\def\ENtwo{v}

\def\FigR#1 #2 {Fig.~\ref{#1}-\ref{#2}}

\begin{document}

\title{Sample Path Properties of the Average Generation of a Bellman-Harris Process}
\author{
Gianfelice Meli\thanks{Hamilton Institute, Maynooth University, Ireland}, 
Tom S. Weber\thanks{The Walter and Eliza Hall Institute of Medical Research \& The
University of Melbourne, Parkville, Australia},
and
Ken R. Duffy\thanks{Hamilton Institute, Maynooth University, Ireland.
E-mail: Corresponding ken.duffy@nuim.ie}
}

\date{}

\maketitle

\begin{abstract}
Motivated by a recently proposed design for a DNA coded randomised
algorithm that enables inference of the average generation of a collection
of cells descendent from a common progenitor, here we establish
strong convergence properties for the average generation of a
super-critical Bellman-Harris process. We further extend those
results to a two-type Bellman-Harris process where one type can
give rise to the other, but not vice versa. These results further
affirm the estimation method's potential utility by establishing
its long run accuracy on individual sample-paths, and significantly
expanding its remit to encompass cellular development that gives rise 
to differentiated offspring with distinct population dynamics.
\end{abstract}

\section{Introduction}

Consider a collection of cells subject to proliferation, differentiation
and death. Define the generation of each descendent to be the number
of divisions that led to that cell. Generation dependent behaviour
has been implicated in the risk of cancer and its evolution
\cite{Frank2003Patterns,Merlo2006Cancer,tomasetti15}, as well as
being a determiner in the complex differentiation dynamics of
proliferating cell systems
\cite{hodgkin1996b,Tangye03,turner08,hills2009,Duffy12,zhang2013,deboer2013,marchingo2014}.
If a cell population expands asynchronously or is subject to death
as well as division, then the average generation of a collection
of cells cannot be inferred solely from knowledge of cell numbers,
Fig. \ref{fig:1a}, and additional information is needed to determine
this quantity Fig. \ref{fig:1b}.

A range of experimental techniques have been developed that allow
evaluation or estimation of the generations of cells. 
Entire lineages can be recorded {\it in vitro} via time lapse
microscopy so long as cells can be continuously tracked,
e.g.
\cite{Powell55,Smith73,Sulston83,Hawkins09,Gomes11,Giurumescu2012Quantitative,richards2013}.
An alternate methodology is to stain initial cells with a fluorescent
dye \cite{lyons94,lyons00,hawkins07,quah12} such that with each
division cells inherit approximately half of the molecules from
their parent and thus fluoresce with half their intensity. A cell's
generation can thus be inferred from its luminous intensity via
flow cytometry. This high throughput approach is suitable for adherent
cells that cannot be tracked optically, and can be used {\it in vivo}
 adoptive transfer experiments. In most applications division tracking
dyes are used to determine the
distribution of a population across generations, but recent
developments have created an experiment design where the offspring
of individual clones can be identified via colour multiplexes of
distinct division diluting dyes \cite{marchingo2016,horton18}. Genetically
modified mice also exist that enable an inducible equivalent of
a division diluting dye {\it
in vivo} without the need for adoptive transfer of {\it ex-vivo}
stained cells, e.g. \cite{tumbar2004defining,foudi2009analysis,mascre2012distinct}.  These methods
enable 6-10 generations to be followed before fluorescent
signal-to-noise ratio is too low for a cell's generation to be
reliably determined.

Methods to estimate replicative tree depth {\it in vivo} have been proposed
that involve measurement of average telomere length
\cite{harley1990,allsopp1992,vaziri1994,weinrich1997,rufer1999,hills2009}
or the number of somatic mutations introduced during
DNA duplication
\cite{shibata96,tsao00,shibata06,Wasserstrom08,reizel11,carlson2012}.
Methods in this direction rely on inference rather than direct
determination, but they offer the possibility of tracing more than
10 generations {\it in vivo}.

We recently proposed a new design for {\it in vivo} inference of
average generation that relies on a DNA coded randomised algorithm
\cite{weber2016inferring}.  For illustration, consider a single
initial cell at time $t=0$. As in Figs.~\ref{fig:1a} and~\ref{fig:1b}, let $Z(t)$ be
the number of offspring alive at time $t$ and $G(t)$ be the sum of
the generations of all living cells at that time. The proposal to
infer $G(t)/Z(t)$ in \cite{weber2016inferring} is to equip the
initial cell with a neutral label, i.e. one whose presence or absence
has no ramifications for population dynamics, such that during each
cell's lifetime with a small probability $p$ the label is irrevocably
and heritably lost. With $\Zp(t)$ denoting the number of label positive cells at
time $t$, as in Fig.  \ref{fig:1c}, the suggested estimator is
\begin{align}
\frac{G(t)}{Z(t)}
\approx -\frac{1}{p}
\log\left(\frac{\Zp(t)}{Z(t)}\right), \text{ assuming } p \text{ is small}.
\label{eq:theformula}
\end{align}
This surprising formula is desirable for a number of reasons: 1)
it allows for cell death; 2) it does not require knowledge of cell
cycle times; and 3) for inference it requires only a proportional
measurement rather than absolute numbers. Moreover, to infer the
relative developmental depth of two populations equipped with the
system, one does not need to know $p$, the probability of label
loss per cell lifetime, if it is the same for both. A DNA coded
randomised algorithm, based on the existing FUCCI cell cycle reporter
\cite{sakaue2008visualizing}, to realise the design is proposed in
\cite{weber2016inferring}.

Two distinct derivations of the approximation \eqref{eq:theformula}
are provided in \cite{weber2016inferring}. One, based on properties
of cumulant generating functions, establishes that for an arbitrary lineage
relationship between the cells constituting $Z(t)$, the expected
number of label-positive cells, $\E(\Zp(t))$, over all possible
delabellings recovers the correct value as the probability of label
loss goes to zero:
\begin{align*}
\frac{G(t)}{Z(t)}
=\lim_{p\to0} -\frac{1}{p}
\log\left(\frac{\E(\Zp(t))}{Z(t)}\right).
\end{align*}
For a single realisation of the delabelling process, as would occur
experimentally, this provides no assurance. To establish
such a result, some structure is needed on the family tree.
Consequently, a complementary result is also established in
\cite{weber2016inferring} within the context of the standard model
of an asynchronously developing tree, the Bellman-Harris branching process.
That is, a growing tree model where cells have i.i.d. lifetimes and
independent i.i.d. numbers of offspring numbers at the end of their
lives. With $Z(t)$ being number of cells alive at time $t$ in a
super-critical 
Bellman-Harris branching process, so long as the label-positive
sub-population $\Zp(t)$ is super-critical,
it is established in \cite{weber2016inferring} that
\begin{align} \label{eq:Ken_formula}
\lim_{t\to\infty}
\frac{\E(G(t))}{t\E(Z(t))} 
	=\lim_{p\to0}\lim_{t\to\infty}
	-\frac{1}{pt} \log\left(\frac{\Zp(t)}{Z(t)}\right), 
	\text{ almost surely if } \liminf_{t\to\infty} \Zp(t)>0.  
\end{align} 
The right hand side of this equation says that as long as the
label-positive sub-population continues to exist, ultimately the estimate
of average generation converges on each single path of the process.
The left hand side, however, is not entirely satisfactory. It is
an average quantity over realisations of the branching process and
it forms the ratio of expectations, $\E(G(t))/\E(Z(t))$, rather
than the expectation of the ratio $\E(G(t)/Z(t)))$.

In the present paper we make two mathematical advances
that further enhance the promise of the proposed method. We first
rectify this shortcoming by proving a substantially stronger result:
that for a Bellman-Harris branching process the sample-path average
generation divided by time converges almost surely to a constant,
giving
\begin{align} \label{eq:main_result}
\lim_{t\to\infty}
\frac{G(t)}{tZ(t)} =\lim_{p\to0}\lim_{t\to\infty} -\frac{1}{pt}
\log\left(\frac{\Zp(t)}{Z(t)}\right), \text{ almost surely if }
\liminf_{t\to\infty} \Zp(t)>0.  
\end{align} 
The convergence result on the left hand side greatly strengthens
the only previous result we are aware of, that proved in \cite{Samuels71}
where convergence in probability of average generation is established
for processes in which there is no death. Given the ubiquity of
Bellman-Harris processes, it is likely to be of interest for other
reasons, but for our purposes it is most significant in providing
extra support for merits of the proposed average generation inference
methodology.

In order to establish this fact we prove a collection of 
surprising results for the paired processes $(Z(t),G(t))$ of a
super-critical Bellman-Harris process. In particular, with $L$
being a lifetime distribution, $\EN>1$ being the average number of
offspring of a cell at the end of its life and $\alpha$ being the
Malthusian parameter, i.e. the solution to
\begin{align}
\label{eq:malthus}
\EN \E(e^{-\alpha L})=1,
\end{align}
then 
\begin{align}
\label{eq:BHW}
\lim_{t\to\infty} \left(e^{-\alpha t} Z(t), t^{-1} e^{-\alpha t} G(t)\right) 
= (c_1\WZ,c_2\WZ),
\end{align}
where $\WZ$ is a random variable and $c_1,c_2$ are constants. Namely, even
though the total generation advances at a different rate to the
population size, the random element of the prefactor is the
same for both, and properties of the ratio $G(t)/Z(t)$ follow.

To establish those results we use a combination of both old and
novel arguments, essentially following the methodology described
by Harris~\cite{harris1963theory}, but relying on a peculiar renewal
theorem for defective measures inspired by results of
Asmussen~\cite{asmussen1998probabilistic}. That allows us to obtain
an integral formulation for the probability generating functions
of the prefactors described above. To clinch the result, we
essentially insert the guess that the randomness in the prefactors
of the two processes is the same.

\begin{figure}
\begin{center}
\subfigure[\label{fig:1a}]{
\includegraphics[height=6cm]{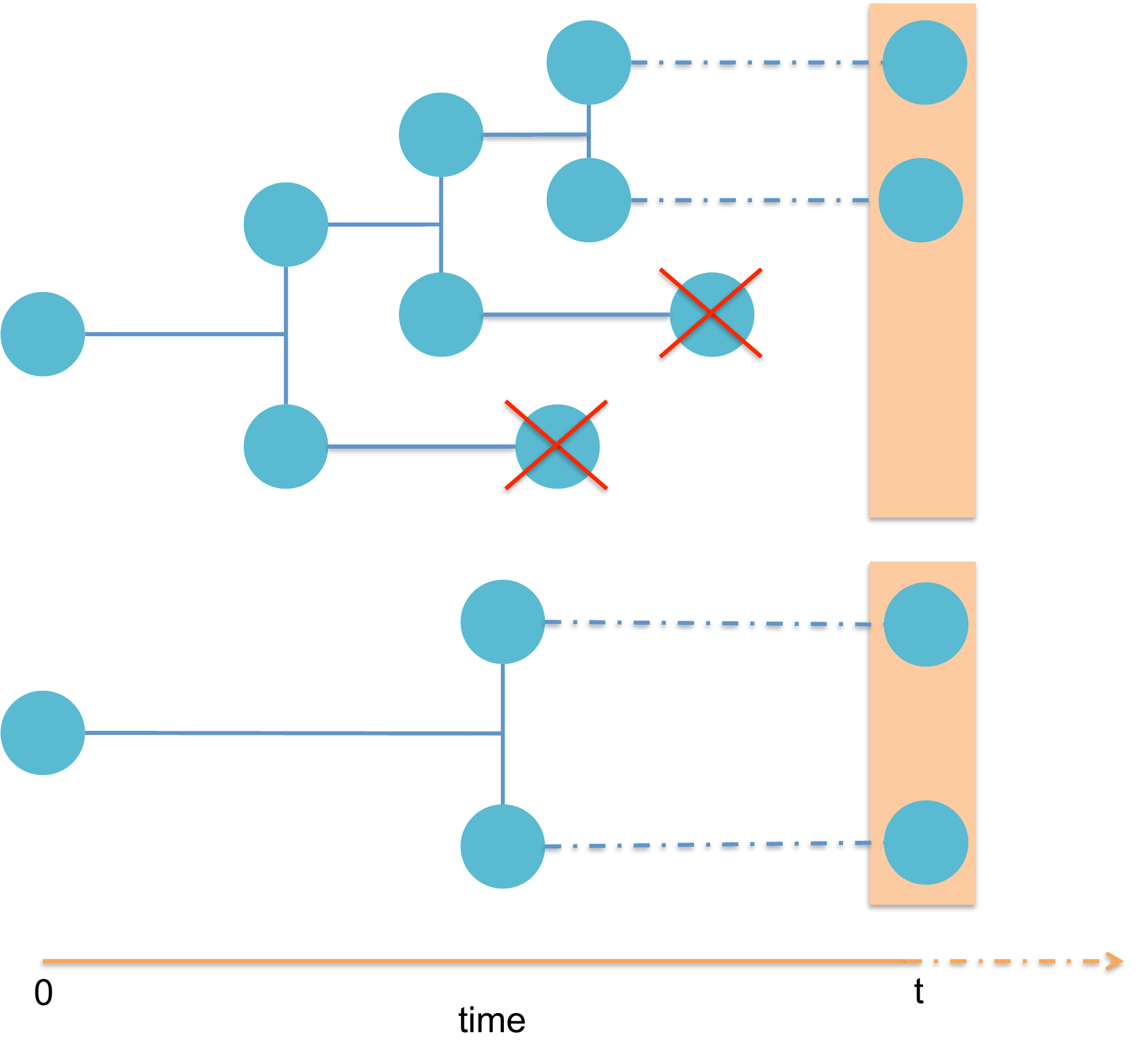}
}\\
\subfigure[\label{fig:1b}]{
\includegraphics[height=6cm]{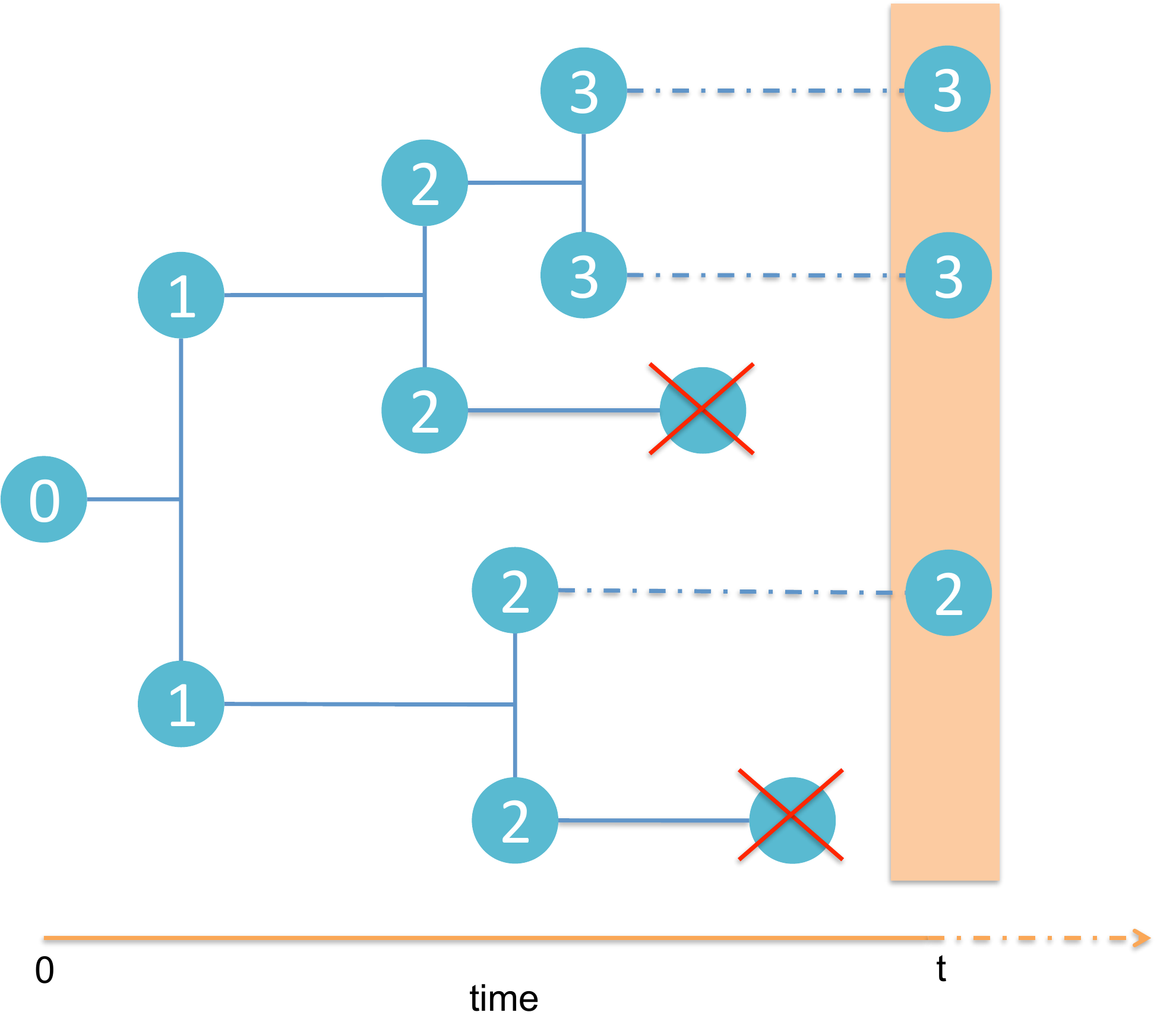}
}
\subfigure[\label{fig:1c}]{
\includegraphics[height=6cm]{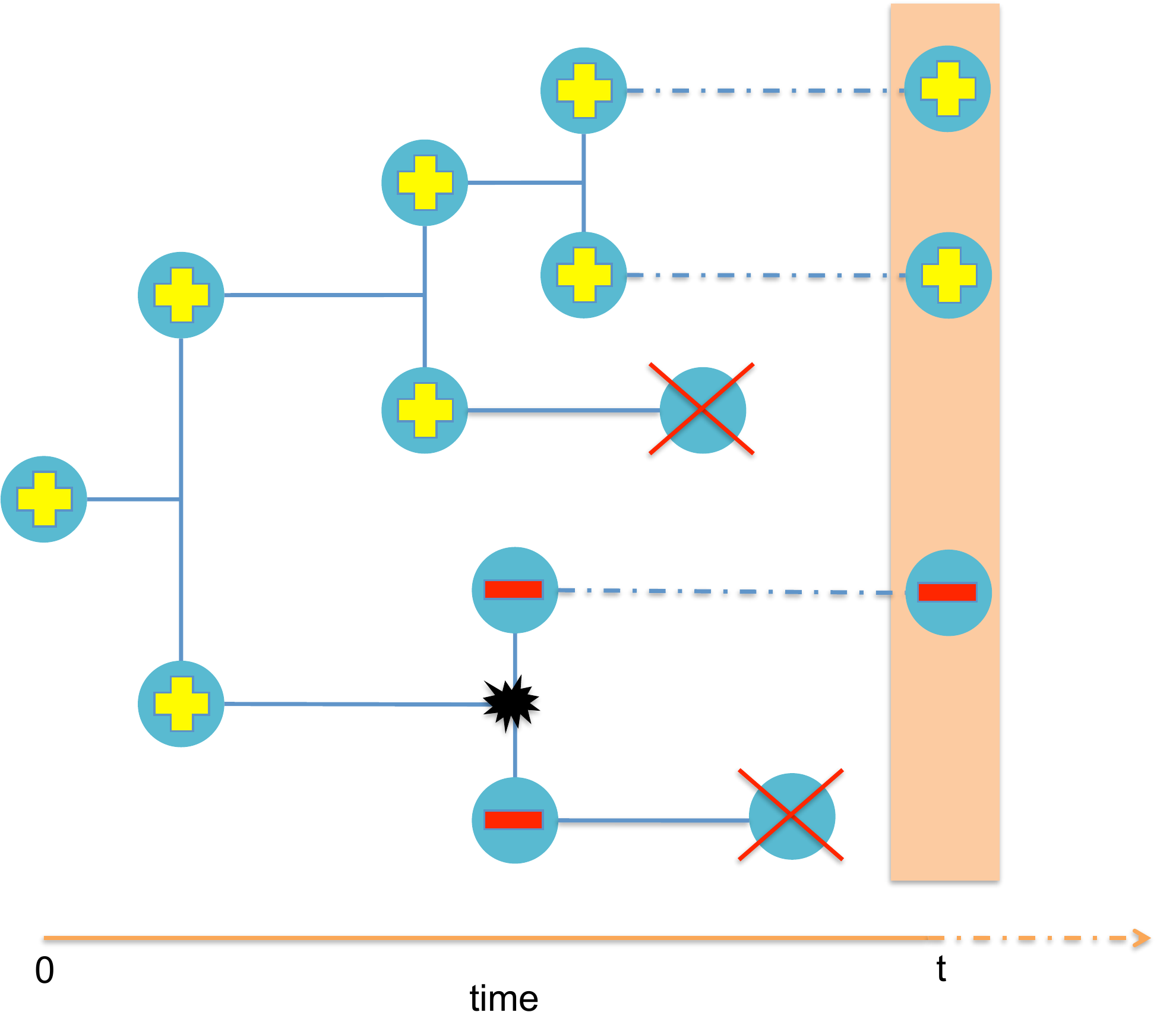}
}
\end{center}
\caption{{\bf Average generation.} (a) If a population of cells
grows asynchronous or is subject to death, knowledge of the number of
cells alive at a single time (orange box, time $t$, $Z(t)=2$) does not uniquely determine
the average number of divisions that lead to to the living cells
(i.e. the depth of the family tree).  (b) With the progenitor being
defined to be in generation 0, the total generation of the process
at any time is the sum of the generations, the number of edges back
to the root of the tree, of living cells (orange box, $G(t) =
3+3+2=8$) and the average generation is the total generation
divided by the number of living cells, $G(t)/Z(t) = 8/3$.
(c) The randomised algorithm proposed in \cite{weber2016inferring}
for inferring $G(t)/Z(t)$ is based on having a neutral label in the
initial cell that is independently lost with probability $p$ during
each cell's lifetime (indicated by the black cloud) and is not regained
by further offspring once lost. If the proportion of label-positive
cells can be measured and the probability of label loss, $p$, is
small, then the following relationship holds $G(t)/Z(t)\approx -1/p
\log(\Zp(t)/Z(t))$ in two approximate senses more fully explained in
the main text.} \label{fig:1} \end{figure}

The second contribution of the present paper is to provide mathematical
support that significantly extends the remit of the average generation
estimation scheme by considering a two-type super-critical
Bellman-Harris process with one-way differentiation, where cells
of the first type can differentiate into cells of the second, but
not vice versa (e.g.  Fig.~\ref{fig:cell_population}).  Assuming,
as before, a cell of the first type is equipped with a neutral label
that is heritably lost at each division with a given probability,
we establish that a relationship akin to that given in
\eqref{eq:main_result} holds for both cell types irrespective of
the ordering of their Malthus exponents. Namely, if one starts with
a single cell of one type that can differentiate and generate a
second type, one can ultimately drawn inferences about the average
generation of each cell type. This encompasses, for example, scenarios
where healthy cells may give rise to quickly growing cancer cells
or quickly expanding multipotent progenitors give rise to slowly
dividing terminally differentiated cells.

\begin{figure}
\begin{center}
\includegraphics[scale=0.45]{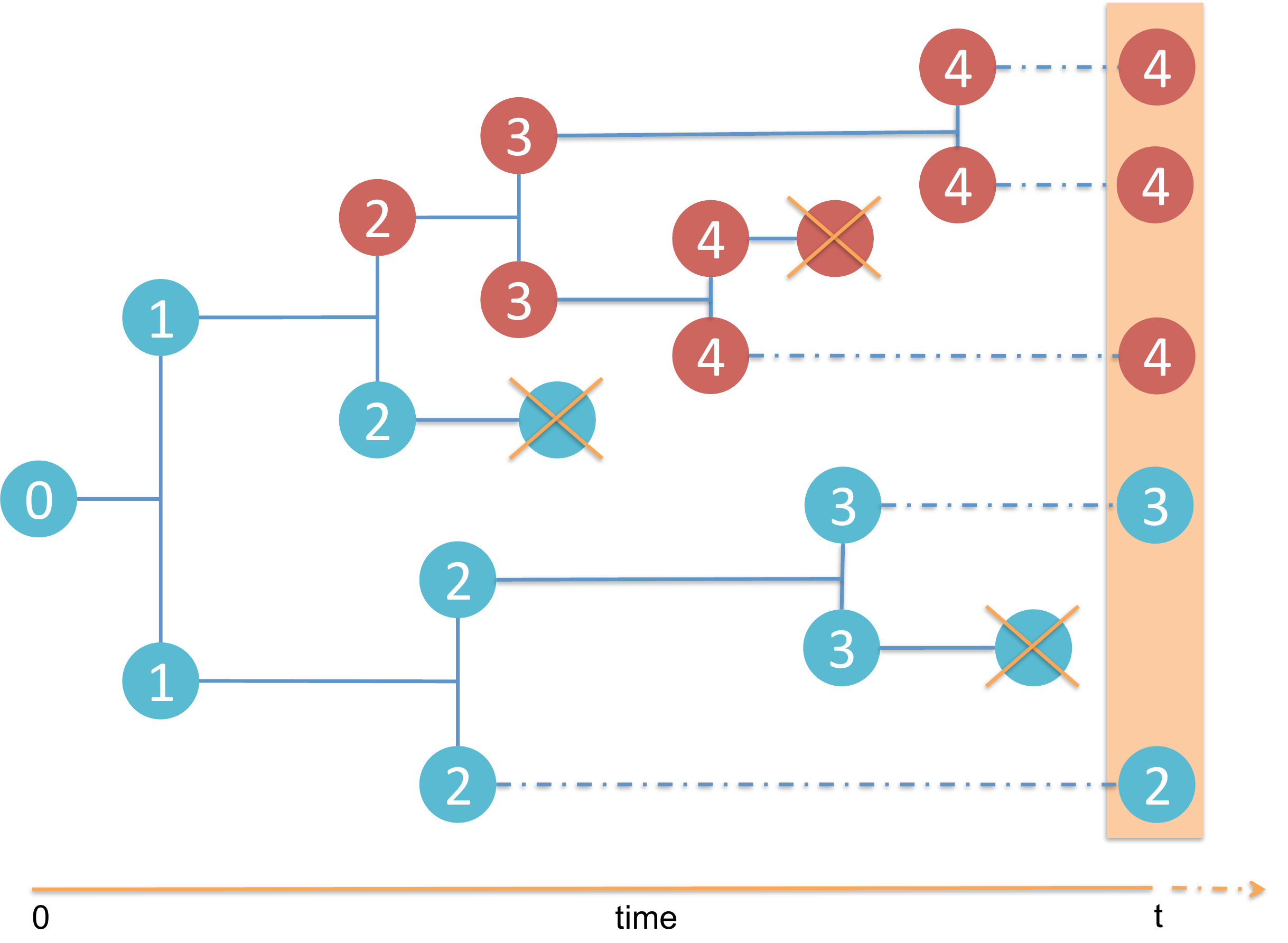}
\end{center}
\caption{\textbf{Two-type process.} In addition to division and
death, a cell may differentiate into another type (indicated here
by a change in colour) with distinct proliferation properties. 
For many scientific questions, one is interested in the average
generation of cells of each type.  The figure describes the growth
of a population that starts with one cell of type-1 at time $0$
and, after consecutive divisions, consists of 5 cells at time
$t$. The average generation of cells of type-1, the blue cells, is
$(3+2)/2=2.5$, while for cells of type-2, the red cells, it is $4$.}
\label{fig:cell_population} 
\end{figure}

\section{Motivation for the main mathematical result}

A time-dependent model of a family tree is necessary to investigate the
temporal dyamics of average generation. Analysis is trivial in the simplest
such stochastic model, the Galton-Watson branching process
\cite{watson1875probability,harris1963theory,Kimmel02}. It assumes
that all cells of a given generation share a common lifetime at the
end of which they produce i.i.d. numbers of offspring for the next
generation. If $t_n$ is the time of birth of the $n^\text{th}$
generation, then the total generation is
simply $G(t_n)=n Z(t_n)$.
Consequently, the well known result for the limit behaviour of
$Z(t_n)$ as $n$ becomes large in the super-critical case~\cite[Chapter
1]{harris1963theory} also describes the prefactor on front of the
distribution of $G(t_n)$,
\begin{align}\label{eq:GW_Z_and_G}
    \lim_{n \to \infty} \frac{Z(t_n)}{h^{n}}= \WZ   \implies 
    \lim_{n \to \infty} \frac{G(t_n)}{nh^{n}}= \WZ 
\end{align}
where $\EN>1$ is the average number of offspring, $\WZ$ is a
non-negative random variable such that $\E(\WZ)=1$, and the equalities in~\eqref{eq:GW_Z_and_G} are meant in distribution. 

On relaxing the constraint that all lifetimes are equal, however,
there seems to be little {\it a priori} reason why the analogous
quantity to $\WZ$ in \eqref{eq:GW_Z_and_G}, which is $\WZ$ in \eqref{eq:BHW},
should be shared by both $Z(t)$ and $G(t)$. Moving away from
synchronicity, if the lifetimes of cells are i.i.d.  positive and
non-lattice random variables, the development forms a
Bellman-Harris branching process~\cite{harris1963theory,Kimmel02}.
In that setting, cells are spread across generations and the ratio
$G(t)/Z(t)$ is no longer deterministic. As $\E(G(t))/(t \E(Z(t)))$
converges to a constant~\cite{weber2016inferring}, it is reasonable
to suspect that the average generation will still grow linearly in
time. That possibility is also suggested by
Fig.~\ref{fig:comparison_Z_and_G}, where, for independent simulations
of a super-critical Bellman-Harris process with Malthusian parameter
$\alpha$ defined in \eqref{eq:malthus}, $Z(t)e^{-\alpha t}$ and
$G(t)e^{-\alpha t}$ are plotted, illustrating the factor $t$ in the
ratio between them.

Collating observations across multiple simulations, however,
Fig.~\ref{fig:comparison_W_Wtilde} suggests something analogous to
\eqref{eq:GW_Z_and_G} is taking place.
Fig.~\ref{fig:comparison_W_Wtilde}(a) plots the empirical cumulative
distribution function of the renormalised total cell numbers and
total generation at a large time, suggesting equality in distribution.
Fig.~\ref{fig:comparison_W_Wtilde}(b) displays a scatter plot of
the per-simulation prefactors of those quantities for large $t$.
There is a strong positive correlation in these values, hinting at
their relatedness. Finally Fig.~\ref{fig:comparison_W_Wtilde}(c)
shows sample paths of the the difference between the renormalised
total cell numbers less renormalised total generation, which appears
to be converging to zero.  This further suggests convergence in
probability of the sample-path average generation of a Bellman-Harris
process, conditional on survival.  Thus, even though $G(t)/Z(t)$
is not longer deterministic, the randomness in $G(t)/Z(t)$ does not
reside in the linear term, but in something smaller, which is one
result that formally established in this paper.

\begin{figure}
\centering
\subfigure[\label{fig:comparison_Z_and_G_a}]{
\includegraphics[scale=0.4]{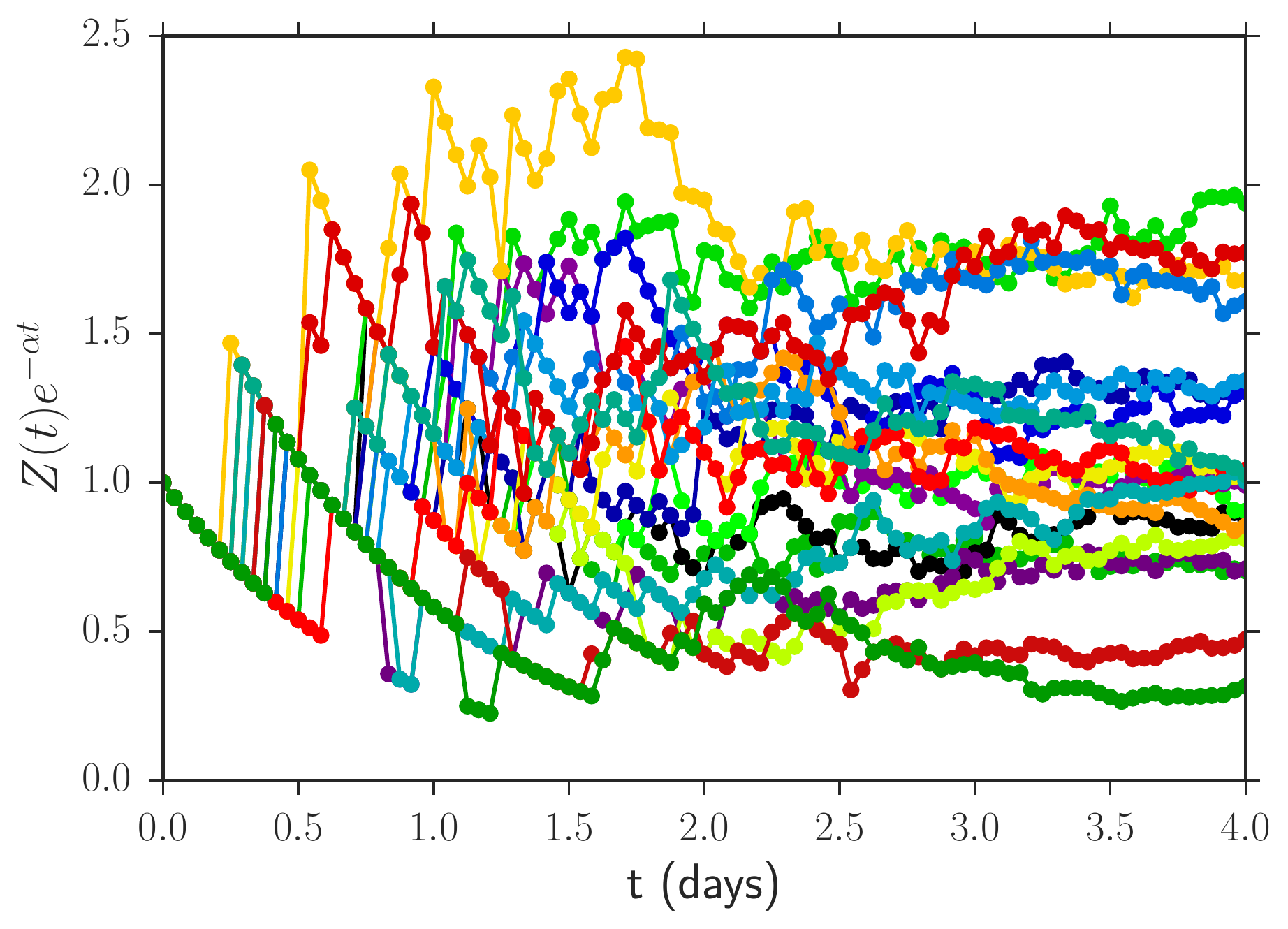}
}
\subfigure[\label{fig:comparison_Z_and_G_b}]{
\includegraphics[scale=0.4]{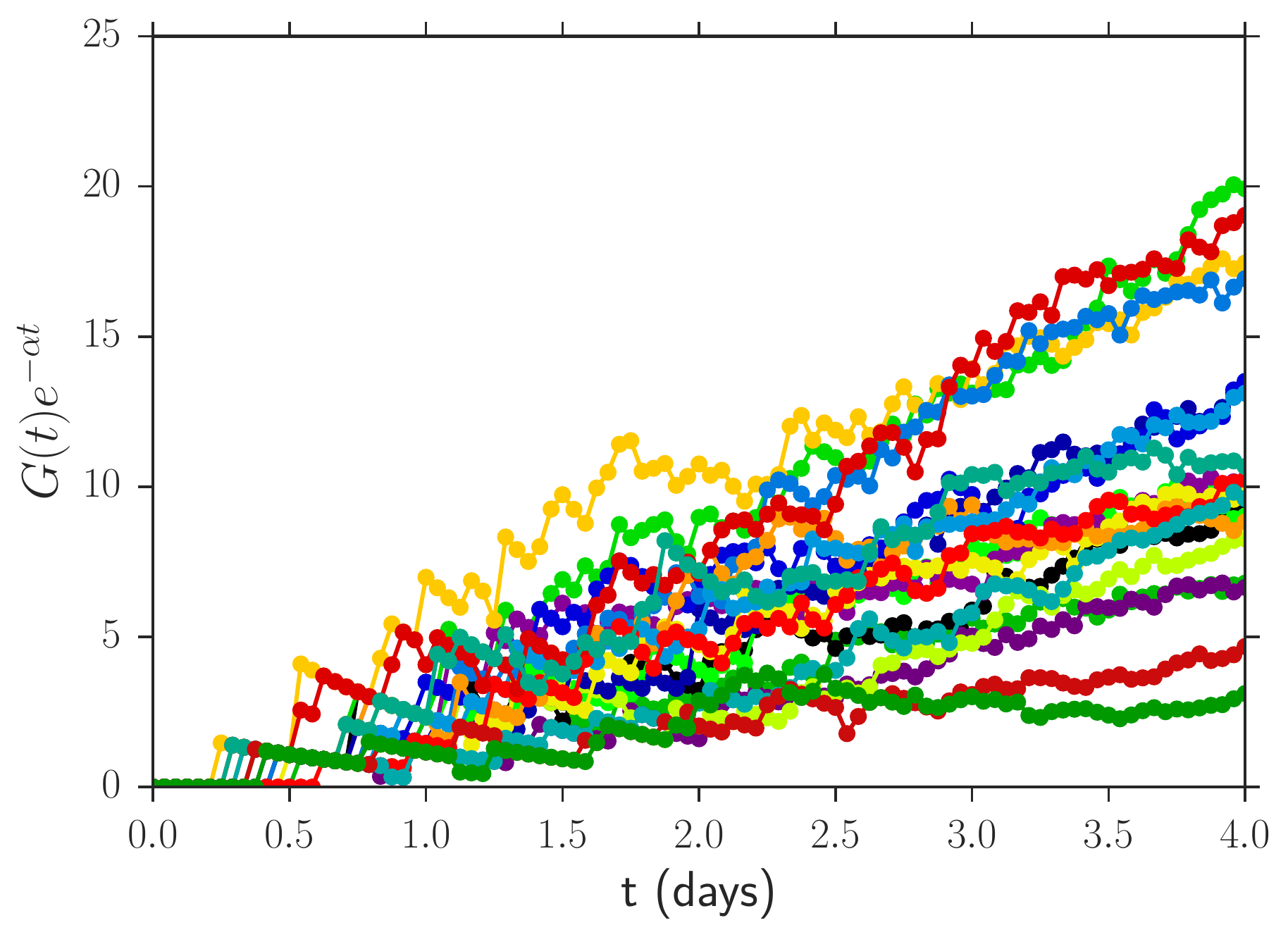}
}
\caption{{\bf Growth rates of population size, $Z(t)$, and total generation, $G(t)$ of a super-critical Bellmann-Harris
process.} 
Each plots present 20 Monte Carlo simulations of a Bellmann-Harris
branching process starting at $t=0$ with a single cell, where paths are
conditioned to have living cells at the final time-point of the simulation. Lifetimes
are lognormal with mean $9.3$ hours and standard deviation $2.54$,
which coincide with those measured for murine B cells stimulated {\it in vitro}
with CpG DNA \cite{Hawkins09}. At the end of each cell's life it gives rise to no
cells
with probability $1/5$ and two with probability $4/5$.  (a) With
$Z(t)$ being the population size at time $t$ and
$\alpha>0$ being the Malthusian parameter defined in equation~\eqref{eq:malthus},
this figure plots the evolution of $Z(t)/e^{\alpha t}$, which is
known to converge almost surely and in mean square to a random
variable $A$, e.g. \cite{harris1963theory}. (b) With $G(t)$ denoting
the total generation of the process (see Fig. \ref{fig:1}) at time
$t$, for the same paths this plot shows $G(t)/e^{\alpha t}$, which
grows linearly over time with a random slope $B$. Results in Section~\ref{ch:MGF W}
establish that $A$ and $B$ are almost surely the same, up
to a multiplicative constant, on a path-by-path basis.}
\label{fig:comparison_Z_and_G}
\end{figure}

\begin{figure}
\subfigure[\label{fig:comparison_W_Wtilde_a}]{
\includegraphics[scale=0.4]{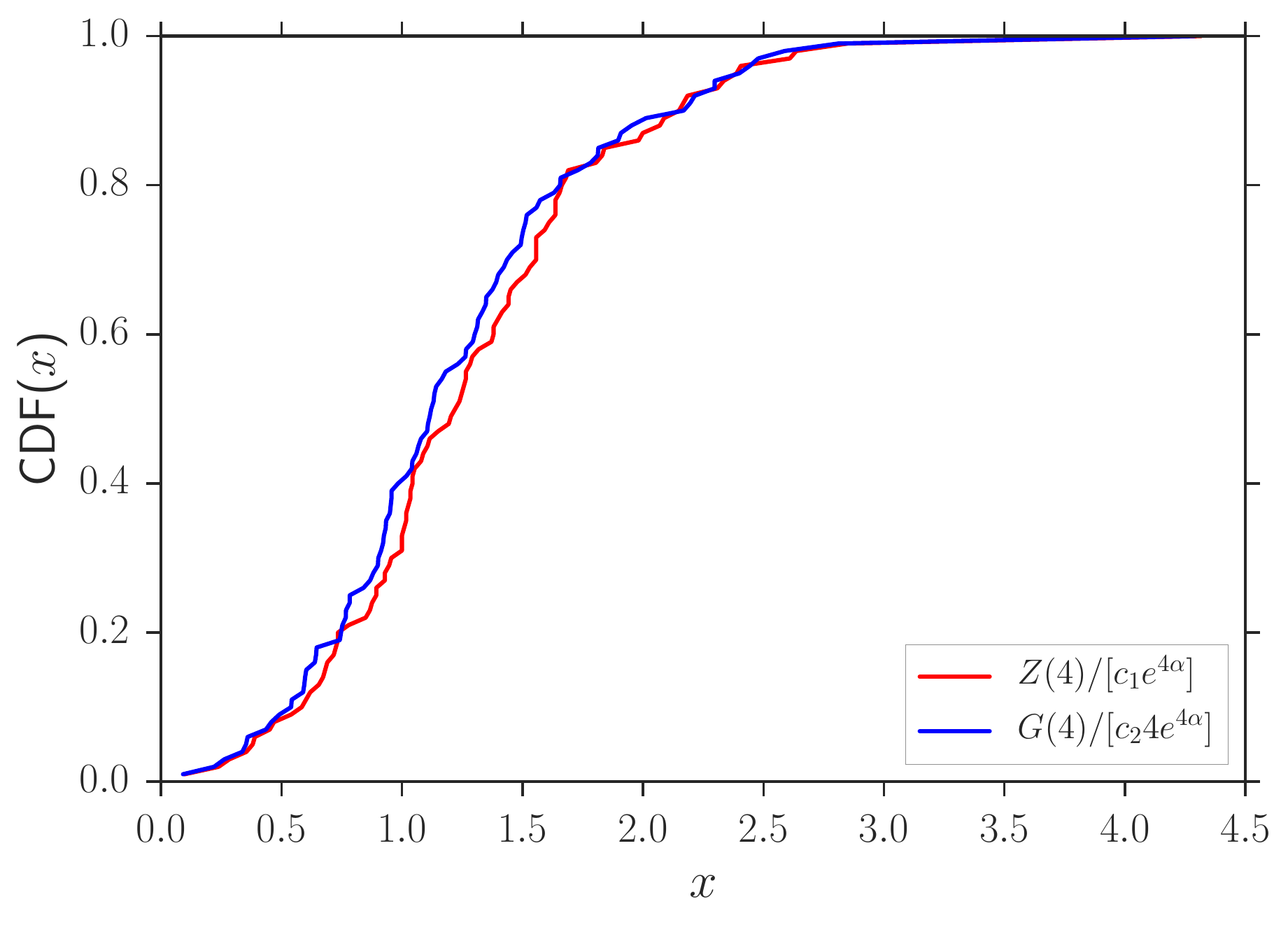}
}
\subfigure[\label{fig:comparison_W_Wtilde_b}]{
\includegraphics[scale=0.4]{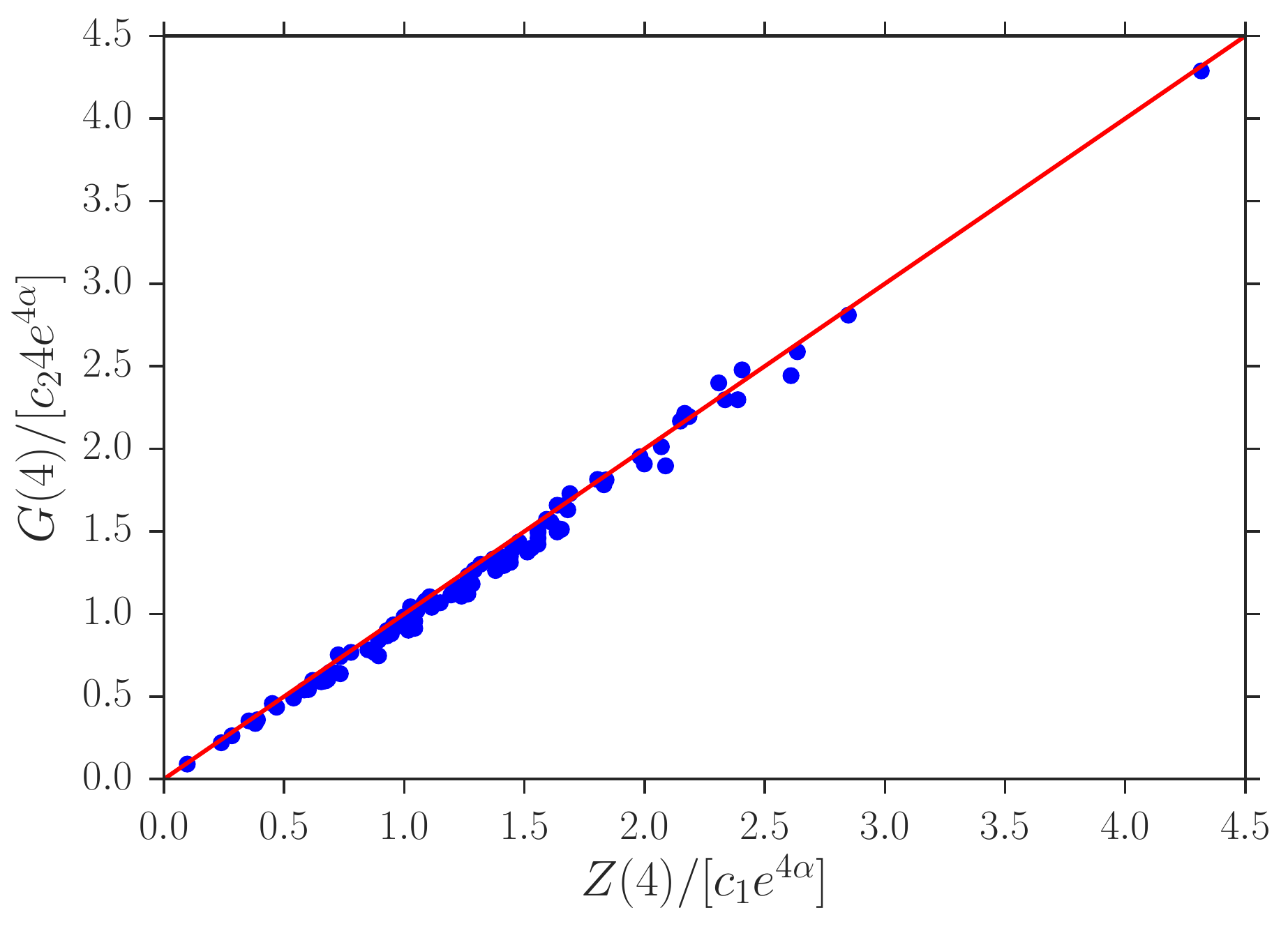}
}
\subfigure[\label{fig:comparison_W_Wtilde_c}]{
\includegraphics[scale=0.4]{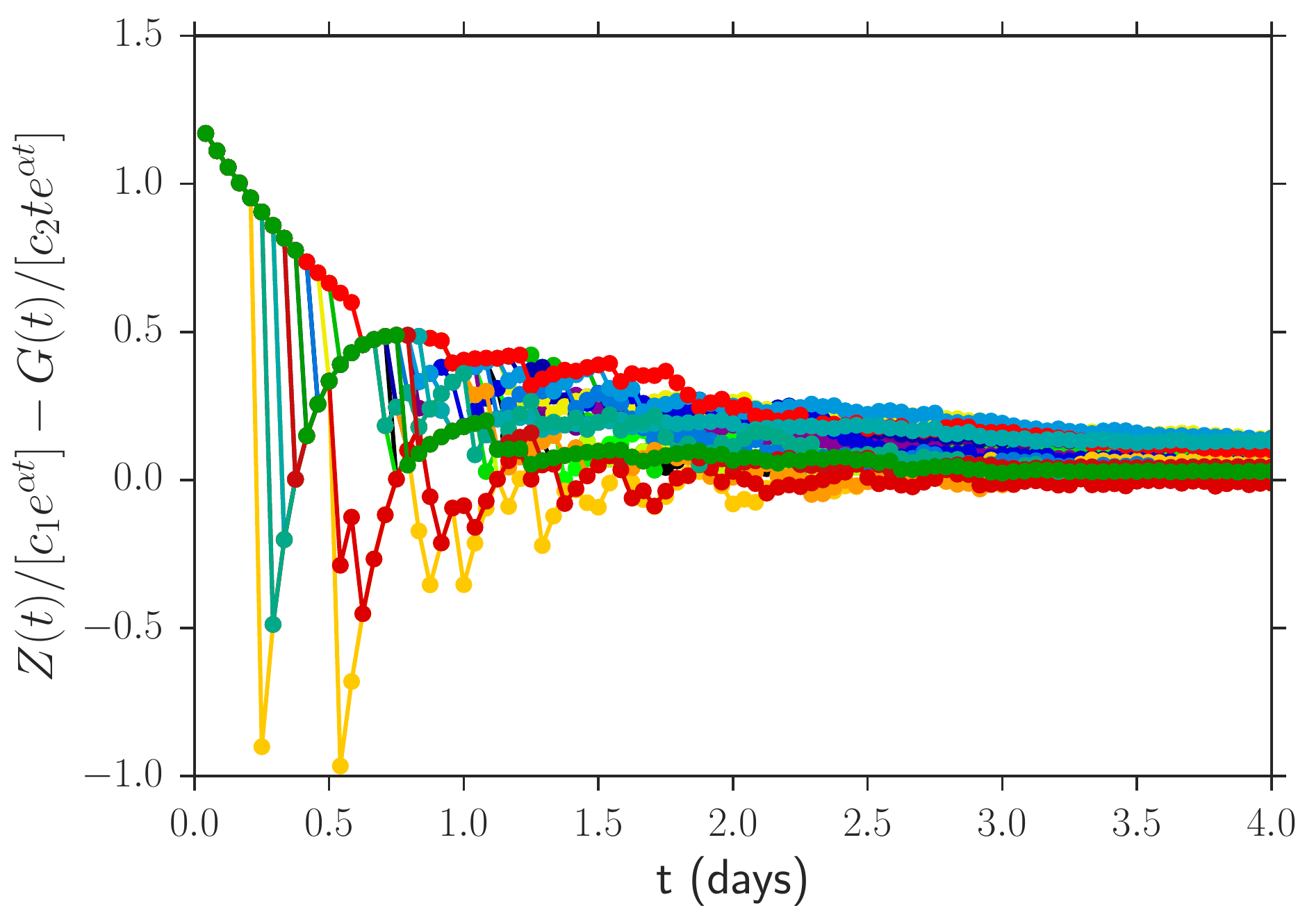}
}
\caption{{\bf Comparison between simulations of $Z(t)/e^{\alpha t}$ and $G(t)/(te^{\alpha t})$.} 
These show results from $100$ Monte Carlo simulations of a Bellman-Harris
process with paramaterization as in Fig.  \ref{fig:comparison_Z_and_G}.
(a) At $t=4$ days, empirical cumulative distribution function (ECDF)
of $Z(t)/(c_1e^{\alpha t})$ and $G(t)/(c_2 te^{\alpha t})$ are
shown, where $c_1$ and $c_2$ are constants that normalise the limit
behaviour of means of the two processes and are computed numerically. 
The ECDFs of the prefactor
on the population size and the slope of the total generation process
are similar suggesting that they follow the same distribution.  (b)
Also at $t=4$ days, the scatter plot of $Z(t)/(c_1e^{\alpha t})$
versus $G(t)/(c_2 te^{\alpha t})$ on a path-by-path basis suggests a
stronger result, that there is equality almost surely.  This
impression is further informed by plot (c) where $20$ paths describing
the evolution over time of  $Z(t)/(c_1e^{\alpha t})-G(t)/(c_2
te^{\alpha t})$, which appear to converge to zero as $t$ increases,
are displayed.}
\label{fig:comparison_W_Wtilde}
\end{figure}

\section{Convergence of the normalised average generation of a super-critical Bellman-Harris process}  
\label{sec:main}
\subsection{Assumptions, notation and previous results}\label{ch:assumption}

The following notation and assumptions are in force throughout
Section \ref{sec:main}. We consider a Bellman-Harris branching
processes with strictly positive non-lattice lifetime random variable
$L$ and non-negative offspring random variable $N$. We define
$\EN:=\mathbb{E}(N)$ and $\ENtwo:=\mathbb{E}(N(N-1))$, and assume
that both are finite. We work within the super-critical case,
$\EN>1$, so that the population has a positive probability of
escaping extinction~\cite{harris1963theory}.

We make use of the Malthusian parameter $\alpha$ defined
in~\eqref{eq:malthus}. As $\EN>1$, $\alpha>0$ exists and is unique.
For $\EN>1$, it is established in Proposition 1 of
\cite{weber2016inferring} that the Malthusian parameter $\alpha$
is a real analytic function of $\EN$. For our purposes, we don't
need to consider $\alpha$ as a function of $\EN$, but we will
sometimes use the notation $\alpha'$ to indicate the value
$d\alpha(x)/dx|_{x=\EN}$. To study the limit behaviour of the scaled
version of the process $(Z(t),G(t))$ we use standard notions of
convergence in distribution (D), in mean square ($L^2$), and almost
surely (a.s.)~\cite{Rudin76,feller68}. Convolution between functions
will be denoted by the operator $*$.
Occasionally in the text we will refer to the underlying measurable
space or the probability space, which we denote as $(\Omega,
\mathcal{B}(\Omega))$ and $(\Omega, \mathcal{B}(\Omega), \mathbb{P})$,
respectively. Example constructions of such spaces can be found
in~\cite[Chapter VI.2]{harris1963theory}.

A brief summary of known results concerning $Z(t)$ and $G(t)$ will
follow. According to~\cite{harris1963theory,jagers1969renewal},
under the above assumptions, the limit behaviour of $Z(t)$ satisfies
\begin{equation}\label{eq:Z(t)}
    \frac{Z(t)}{e^{\alpha t}} \xrightarrow[]{a.s., L^2} c\WZ,
\end{equation}
where $\WZ$ is a non-negative random variable such that $\E(\WZ)=1$, and 
\begin{align*}
    c= \lim_{t \to \infty}\frac{\E({Z}(t))}{e^{\alpha t}}= 
	\frac{\int_0^\infty \mathbb{P}(L>t) e^{- \alpha t} dt}{h \int_0^\infty u e^{- \alpha u} d\mathbb{P}(L \leq u)}= \frac{h-1}{h^2 \alpha \int_0^\infty u e^{- \alpha u} d\mathbb{P}(L\leq u)}. 
\end{align*}
For the expected value of $G(t)$, the following is proven in Theorem 2 of ~\cite{weber2016inferring}
\begin{align}\label{eq:E(G(t))}
	\lim_{t \to \infty}
    \frac{\mathbb{E}(G(t))}{t e^{\alpha t}} 
	= h \alpha' c,\,
\text{where }
\alpha' = \frac{1}{h^2 \int_0^{+ \infty}u e^{- \alpha u}d\mathbb{P}(L\leq u)}.
\end{align}
There, we find also information concerning the asymptotic covariance
of $G(t)$ and $Z(t)$ and the ratio of their expectations,
\begin{align}
	\lim_{t \to \infty}
  \frac{\mathbb{E}(G(t)Z(t))}{t e^{2\alpha t}} = c^2 h \alpha' k 
	\text{  and  }
	\lim_{t \to \infty}
  \frac{\mathbb{E}(G(t))}{t \mathbb{E}(Z(t))} = h \alpha',
\text{ where } 
k=\frac{v \int_0^\infty e^{- 2 \alpha u} d\mathbb{P}(L\leq u)}{1 - h \int_0^\infty e^{- 2 \alpha u} d\mathbb{P}(L\leq u) }.\label{eq:k} 
\end{align} 

The scaling of means in equations~\eqref{eq:Z(t)} and~\eqref{eq:E(G(t))}
suggest the definition of normalised versions of the processes
$Z(t)$ and
$G(t)$, 
\begin{align}\label{eq:W_and_V}
\WZtime{t}:=\frac{Z(t)}{c e^{\alpha t}} \text{ and } 
	\WGtime{t}:=\frac{G(t)}{c h \alpha' t e^{\alpha t}}, 
\end{align}
whose use will simplify notation in the proofs.

In order to establish one of the main results of the paper, equation~\eqref{eq:main_result},
stated in Corollary \ref{cor:main} of Section~\ref{ch:almost sure
convergence}, we study the limit behaviour of the process
$\{\WGtime{t}\}$. We do that in two steps: first, in Section~\ref{ch:mean
square convergence} we consider $\{\WGtime{t}\}$ as an $L^2$ process
and determine its mean square limit; then, in Section~\ref{ch:almost
sure convergence} we reinforce that result by proving that the
convergence is also valid with probability $1$ under a condition
on the speed of $L^2$ convergence. In Section~\ref{ch:mean square
convergence}, we make extensive use of a particular version of Key
Renewal Theorem for defective measures that we establish in
Section~\ref{ch:lemmas}. Once we prove in Section~\ref{ch:MGF W}
that $\WGtime{t}$ and $\WZtime{t}$ share the same random prefactor
on front of their dominant term for large $t$, we are finally able
to characterise the limit behaviour of $G(t)/(tZ(t))$.

\subsection{A new Renewal Theorem for Defective Measures}\label{ch:lemmas}

In order to prove~\eqref{eq:E(G(t))} in~\cite{weber2016inferring},
a version of the Renewal Theorem due to Asmussen, Theorem 6.2(b)
of \cite{asmussen1998probabilistic}, is used in a fundamental way.
In this section we generalise that theorem to make it applicable
for defective measures, i.e.  measures with total mass less than one. 
Before going to the main result of the section,
Theorem~\ref{thm:renewal theorem}, we first state a non-standard
version of the classical Dominated Convergence Theorem (DCT),
which can be applied to a collection of sequences of functions $\{(f_{t,\tau})_{t\in \mathbbm{R}_{\geq 0}} : \tau \in \mathbbm{R}_{\geq 0} \}$,
each one converging pointwise, when $t \to \infty$, to a same function $f$, uniformly for $\tau \geq 0$. This can be proved essentially repeating the same steps of the classical DCT, including the use of Fatou's lemma, but this time the hypothesis of the uniformity in $\tau$ allows a stronger conclusion. 
This proposition is followed by a lemma that depends on it.

\begin{proposition}[Non-standard DCT]\label{cor:DCT}
Let $(\mathbbm{R},\mathcal{B}(\mathbbm{R}),\mu)$ be a measure space,
and for every $\tau \geq 0$ let $(f_{t,\tau})_{t \geq 0}$ be a
sequence of functions in $L^1(\mu)$  that converges pointwise to
$f$ uniformly for $\tau \in [0,\infty)$, i.e. given $\epsilon>0$ and $u \in \mathbbm{R}$ there exists a $t_{\epsilon, u}>0$ s.t. for every $t\geq t_{\epsilon, u}$ and $\tau \geq 0$ we have $|f_{t,\tau}(u)- f(u)|<\epsilon$. 
Assume there is $g \in L^1(\mu)$ s.t.
$|f_{t,\tau}(u)|\leq g(u)$ for every $t, \tau$, and $u$. Then, $f
\in L^1(\mu)$ and
\begin{equation*}
    \lim_{t\to \infty} \int_{\mathbbm{R}} f_{t,\tau}(u) d\mu(u) = \int_{\mathbbm{R}} f(u) d\mu(u) \qquad \text{ uniformly for $\tau \geq 0$,}
\end{equation*}
i.e. given $\epsilon>0$ there exists a $t^*_{\epsilon}>0$ s.t. for every $t\geq t^*_{\epsilon}$ and $\tau \geq 0$ we have $|\int_{\mathbbm{R}} f_{t,\tau}(u) - f(u) d\mu(u)|<\epsilon$.
\end{proposition}

We are now going to use this version of the DCT to study the limit
behaviour of convolutions between functions and probability measures.
We are interested in these particular structures because we will
show that the moments of $G(t)$ can be written in that form.

\begin{lemma}[Convolution with a finite measure doesn't change convergence rates]\label{lem:convergenza uniforme}
Consider $f=f(t,\tau): \mathbb{R}_{\geq 0} \times \mathbb{R}_{\geq
0} \to \mathbb{R}$ locally bounded in $t$ and s.t., for every $\tau \geq 0$, $f(t,\tau)/[t^p(t+\tau)^q] \to c_1 $ when $t\to \infty$, with $c_1<\infty$, $p,q \geq 0$, and let $\mu$ be a finite measure on $( \mathbb{R}_{\geq 0},
\mathcal{B}(\mathbb{R}_{\geq 0}))$. Then, for every $\tau\geq 0$
\begin{equation}\label{eq:integral equation_2}
    \lim_{t\to \infty}\frac{1}{t^p(t+\tau)^q} \int_0^t f(t-u,\tau) \mu(du) =c_1 \mu([0,\infty)).
\end{equation}
Furthermore, if $|f(t,\tau)|\leq f_1(t)f_2(t+\tau)$, with
$f_i(t): \mathbb{R}_{\geq
0} \to \mathbb{R}_{\geq
0}$ locally bounded functions for $i\in\{1,2\}$,  $f_1(s)/s^p \to a_1$, $f_2(s)/s^q
\to a_2$, and $f(t,\tau)/[t^p(t+\tau)^q] \xrightarrow[t \to \infty]{}
c_1$ uniformly for $\tau\geq 0$ with $a_1, a_2, c_1 <\infty$ and
$p,q \geq 0$, then~\eqref{eq:integral equation_2} is true uniformly for $\tau \in [0, \infty)$. 
\end{lemma}

\begin{proof}
We only prove the second part of the lemma, as the first part follows from the same rationale with the use of the classical Dominated Convergence Theorem instead of Proposition~\ref{cor:DCT}.

For the following, we extend the functions $f$, $f_1$, and $f_2$ to $\mathbbm{R}\times \mathbb{R}_{\geq 0}$, $\mathbbm{R}$, and $\mathbbm{R}$, respectively, by defining $f(t,\tau)=f_1(t)=f_2(t)=0$ when $t <0$. If we can establish that $|f(t-u,\tau)|/[t^p(t+\tau)^q]
\mathbbm{1}_{[0,t)}(u)$ is bounded by a constant $M$, for every $u
\in \mathbbm{R}$, $\tau \geq 0$, and $t$ sufficiently large, we can
apply the DCT in Proposition~\ref{cor:DCT} and conclude
that equation~\eqref{eq:integral equation_2} holds uniformly for $\tau \in [0, \infty)$.

Given $\epsilon>0$, from the hypotheses made, we know that there exists $u_{\epsilon}>0$
s.t. for every $u\geq u_{\epsilon}$ we have $f_1(u)/u^p \leq
a_1+\epsilon$ and $f_2(u)/u^p \leq
a_2+\epsilon$. Without loss of generality we can suppose $t\geq t_{\epsilon,u}:= \max\{u_{\epsilon}, 1\}$. So, for every $u\in \mathbbm{R}$, we have
\begin{align}\label{eq:f_1_sup}
    0 \leq g_t(u):=\frac{f_1(u)}{t^p}\mathbbm{1}_{[0,t)}(u)&= \frac{f_1(u)}{t^p}\mathbbm{1}_{[0,u_{\epsilon})}(u) + \frac{f_1(u)}{t^p}\mathbbm{1}_{[u_{\epsilon},t)}(u) \leq f_1(u) \mathbbm{1}_{[0,u_{\epsilon})}(u) + \frac{f_1(u)}{u^p}\mathbbm{1}_{[u_{\epsilon},\infty)}(u) \notag \\
    & \leq \sup_{[0,u_{\epsilon})}f_1(u) + a_1+\epsilon = M_1 <\infty,
\end{align}
where in the last equality we have used the fact that $f_1$ is a locally bounded function. From~\eqref{eq:f_1_sup}, we have that $g_t(u)$ is dominated by $M_1$ for every $u\in \mathbbm{R}$ and $t\geq t_{\epsilon,u}$. So, the same will be true for $g_t(-u)$, and for its translation $g_t(t-u)$. A similar reasoning can be done with $f_2$, obtaining
\begin{align*}
  \frac{f_1(t-u)}{t^p}\mathbbm{1}_{[0,t)}(u) \leq M_1, \qquad \frac{f_2(t-u)}{t^q}\mathbbm{1}_{[0,t)}(u) \leq M_2,
\end{align*}
for every $u\in \mathbbm{R}$ and $t\geq t_{\epsilon,u}$.
Remembering that for hypothesis $|f(t,\tau)|\leq f_1(t)f_2(t+\tau)$, for every $u\in \mathbbm{R}$, $t\geq t_{\epsilon,u}$, and $\tau\geq 0$ we have
\begin{equation*}
    \frac{|f(t-u,\tau)|}{t^p(t+\tau)^q} \mathbbm{1}_{[0,t)}(u) \leq \frac{f_1(t-u)}{t^p}\mathbbm{1}_{[0,t)}(u)\frac{f_2(t+\tau -u)}{(t+\tau)^q}\mathbbm{1}_{[0,t+\tau)}(u) \leq M_1M_2=:M
\end{equation*}
That concludes the proof.
\end{proof}

Armed with that lemma, we can now prove the main result of this section.

\begin{theorem}[A defective measure version of Theorem 6.2(b) \cite{asmussen1998probabilistic}] \label{thm:renewal theorem} 
Consider the integral equation 
\begin{equation}\label{eq:integral equation}
    K(t, \tau)=f(t, \tau) + \int_0^t K(t-u, \tau) \rho(du),
\end{equation}
where $K,f: \mathbb{R}_{\geq 0} \times \mathbb{R}_{\geq 0} \to \mathbb{R}$, and $\rho$ is a positive defective measure on $( \mathbb{R}_{\geq 0}, \mathcal{B}(\mathbb{R}_{\geq 0}))$, i.e. $\rho([0,\infty))<1$. 
If $f(t,\tau)$ is locally bounded in $t$ and s.t., for every $\tau \geq 0$, $f(t,\tau)/[t^p(t+\tau)^q] \to c_1 $ when $t\to \infty$, with $c_1<\infty$, $p,q \geq 0$, then for every $\tau \geq 0$
\begin{equation}\label{eq:result_corollary}
    \lim_{t\to \infty}\frac{K(t, \tau)}{t^p(t+\tau)^q}=\frac{c_1}{1-\rho([0,\infty))}.
\end{equation}

Furthermore, if $f$ is s. t. $|f(t,\tau)|\leq f_1(t)f_2(t+\tau)$, with $f_i(t): \mathbb{R}_{\geq
0} \to \mathbb{R}_{\geq
0}$ locally bounded functions, $i\in\{1,2\}$, s.t.
$f_1(t)/t^p \to a_1$, $f_2(t)/t^q \to a_2$, and $f(t,\tau)/[t^p(t+\tau)^q] \xrightarrow[t \to \infty]{} c_1 $ uniformly for $\tau\geq 0$ with $a_1,a_2,c_1<\infty$ and $p,q \geq 0$, then~\eqref{eq:result_corollary} is true uniformly for $\tau \geq 0$.
\end{theorem}

\begin{proof}
From~\cite[Theorem 3.5.1]{resnick2013adventures}, the only solution of~\eqref{eq:integral equation} that is bounded on every finite interval of $t$ has the form 
\begin{equation}\label{eq:K(t)_solution}
    K(t, \tau)=(U*f)_{\tau}(t)=\int_0^t f(t-u, \tau)U(du),
\end{equation}
where $U([0,t))=\sum_{n=0}^\infty \rho^{*n}([0,t))$, $\rho^{*n}([0,t))=(\rho*\rho^{*(n-1)})([0,t))$, and $\rho^{*0}([0,t))=\mathbbm{1}_{[0,\infty)}(t)$. Using Lemma~\ref{lem:convergenza uniforme} and the fact that $U([0,\infty))=1/(1-\rho([0,\infty)))$~\cite[Section 3.11]{resnick2013adventures},  we obtain~\eqref{eq:result_corollary}.
\end{proof}

Thanks to the linearity of integration, we have the following mild generalisation.

\begin{corollary}\label{cor:1}
If in Theorem~\ref{thm:renewal theorem} we substitute the condition
$|f(t,\tau)|\leq f_1(t)f_2(t+\tau)$ with $|f(t,\tau)| \leq \sum_{i=1}^n
f_{2i-1}(t)f_{2i}(t+\tau)$, where $f_i$ are locally bounded functions
s.t. $f_{2i-1}(t)/t^p \to a_{2i-1}$, $f_{2i}(t)/t^q \to a_{2i}$,
$a_{2i-1},a_{2i}<\infty$ for every $1\leq i \leq n$, then the
conclusions of Theorem~\ref{thm:renewal theorem} hold.
\end{corollary}

\subsection{Mean square convergence }\label{ch:mean square convergence}

Equation~\eqref{eq:E(G(t))} states that $\mathbb{E}(\WGtime{t}) \to
1$. A natural question that this result rises is whether there exists a
non-negative random variable $\WG$, s.t. $\mathbb{E}(\WG)=1$, to
which $\WGtime{t}$ converges in mean. Studying the behaviour of the
second moment of $\WGtime{t}$, in Theorem~\ref{th:L2 convergence},
the main result of the section, we will prove something stronger
than that: the convergence is true also in $L^2$. To achieve that
we will need a version, stated in Proposition~\ref{prop:F}, of one
of the results presented in~\cite{weber2016inferring} concerning
the Probability Generating Function (PGF) of $(G(t),Z(t))$, that
better fits our purpose. We use it in Lemmas~\ref{lem:Z(t+tau)G(t)}
and~\ref{lem:G(t+tau)G(t)} where a study of the covariance between
$G(t)$ and $Z(t)$, and of the relation between  different terms of the
total generation process is made. This will lead us to
Corollary~\ref{cor:Cauchy}, which allows us to finally prove
Theorem~\ref{th:L2 convergence}.

\begin{proposition}[A reformulation of Theorem 2 of \cite{weber2016inferring}]\label{prop:F}
For $s_1, s_2, r_1, r_2, t, \tau \in \mathbb{R}_{\geq0}$,  define $F(s_1,s_2,r_1,r_2,t,\tau) := \mathbb{E}(s_1^{G(t)} s_2^{G(t+ \tau)} r_1^{Z(t)} r_2^{Z(t+\tau)})$. Then,  we have
\begin{align} \label{eq:probability generating function}
    F(s_1,s_2,r_1,r_2,t,\tau) =&r_1r_2\mathbb{P}(L>t + \tau) + r_1\int_t^{t+\tau} \rho_N\left(\mathbb{E}\left( s_2^{G(t+\tau-u)} (s_2r_2)^{Z(t+\tau-u)}\right)\right) d\mathbb{P}(L\leq u) \notag\\
    &+ \int_0^{t} \rho_N\Big(F(s_1,s_2,s_1r_1,s_2r_2,t-u,\tau)\Big) d\mathbb{P}(L\leq u),
\end{align}
where $\rho_N(s)=\mathbbm{E}(s^N)$, the probability generating function of the offspring number, $N$.
\end{proposition}

Using Proposition~\ref{prop:F}, we analyse the limiting behaviour of the covariance between $Z(t)$ and $G(t)$.

\begin{lemma}[Limit behaviour of the covariance of $\WGtime{t}$ and $\WZtime{t}$]\label{lem:Z(t+tau)G(t)}
Using the previous notation, we have
\begin{equation}\label{eq:V_tW_t}
    \lim_{t \to \infty} \mathbbm{E}(\WGtime{t}\WZtime{t+\tau})= k = \lim_{t \to \infty} \mathbbm{E}(\WGtime{t+\tau}\WZtime{t}) \qquad \text{ uniformly for $\tau \geq 0$,}
\end{equation}
where $k$ is defined in~\eqref{eq:k}.
\end{lemma}

\begin{proof}
We prove only the first of the equalities in~\eqref{eq:V_tW_t} as the other one can be obtained in a similar way.

Consider the integral equation~\eqref{eq:probability generating function} and take the derivative first for $s_1$, secondly for $r_2$, and then evaluate it at $(1,1,1,1,t,\tau)$. We obtain that
\begin{align*}
    \mathbbm{E}(G(t)Z(t+\tau))=& v\int_0^t \left[\mathbbm{E}(G(t-u))\mathbbm{E}(Z(t+\tau -u)) + \mathbbm{E}(Z(t-u))\mathbbm{E}(Z(t+ \tau -u))\right]d\mathbb{P}(L\leq u) \notag \\
    &+h \int_0^t \mathbbm{E}(Z(t-u)Z(t+\tau -u))d\mathbb{P}(L\leq u) \notag \\
    &+ h \int_0^t \mathbbm{E}(G(t-u)Z(t+\tau-u))d\mathbb{P}(L\leq u),
\end{align*}
where we recall that $h=\mathbbm{E}(N)$ and $v=\mathbbm{E}(N(N-1))$.
Multiplying both sides of this equation by $e^{-\alpha t}e^{-\alpha (t+\tau)}$, and denoting 
\begin{align*}
    K(t,\tau) :=\frac{\mathbb{E}(G(t)Z(t+ \tau))}{e^{\alpha t}e^{\alpha (t+\tau)}}, \quad
    d\mathbb{\overline{P}}(L\leq u):= h e^{- 2 \alpha u} d\mathbb{P}(L\leq u), \quad
    d\mathbb{P}'(L\leq u):=v e^{-2\alpha u}d\mathbb{P}(L\leq u),
\end{align*}
\begin{align}
    f(t, \tau):= &\int_0^t \left[\frac{\mathbbm{E}(G(t-u))}{e^{\alpha (t-u)}}\frac{\mathbbm{E}(Z(t+\tau -u))}{e^{\alpha (t + \tau -u)}} + \frac{\mathbbm{E}(Z(t-u))}{e^{\alpha (t-u)}}\frac{\mathbbm{E}(Z(t + \tau -u))}{e^{\alpha (t + \tau -u)}}\right]d\mathbb{P}'(L\leq u) \notag \\
    &+ \int_0^t \frac{\mathbbm{E}(Z(t-u)Z(t+\tau -u))}{e^{\alpha (t-u)}e^{\alpha (t+\tau - u)}}d\mathbb{\overline{P}}(L\leq u), \label{eq:f_1}
\end{align}
we have that
\begin{equation}\label{eq:K(t,tau)_1}
   K(t,\tau)=f(t,\tau) + \int_0^t K(t-u, \tau) d\mathbb{\overline{P}}(L\leq u). 
\end{equation}

Observe that $\mathbb{\overline{P}}$ is a defective measure. In fact, 
\begin{equation}\label{eq:P_difective}
    \int_{0}^{+\infty} d\mathbb{\overline{P}}(L\leq u)= h  \int_{0}^{+\infty} e^{- 2 \alpha u} d\mathbb{P}(L\leq u) < h  \int_{0}^{+\infty} e^{- \alpha u} d\mathbb{P}(L\leq u) \overset{\eqref{eq:malthus}}{=} 1.
\end{equation}
As $\mathbbm{E}(\WGtime{t}\WZtime{t+\tau})=\mathbbm{E}(G(t)Z(t+\tau))/[h
\alpha' c^2 te^{\alpha t} e^{\alpha (t+\tau)}]$, in order to conclude
the proof, we would like to apply Theorem~\ref{thm:renewal theorem}
at~\eqref{eq:K(t,tau)_1} with $p=1$ and $q=0$. So, we need to prove
that the hypotheses on $f(t,\tau)$ are verified.

Note that $f(t,\tau)$ is the sum of three integrals, where each integrand,
divided by $t$, converges to a constant when
$t\to \infty$, uniformly for $\tau \geq 0$  (see~\eqref{eq:E(G(t))},\eqref{eq:Z(t)}, and~\cite[pg.
145]{harris1963theory}). Furthermore, each of these integrands is
dominated by the product of two locally bounded functions (the
moments of $Z(t)$ and $G(t)$ are locally bounded solutions of
integral equations of the type in equation~\eqref{eq:K(t,tau)_1}, see~\cite[pg. 142]{harris1963theory} and \cite[Theorem 2]{weber2016inferring}), one depending on
$t$ and another one depending on $t+\tau$ (for the last integrand,
use the Cauchy-Schwartz inequality to see it). As these
dominant functions satisfy the hypotheses of Lemma~\ref{lem:convergenza
uniforme} with $p=1$ and $q=0$ (see~\eqref{eq:E(G(t))} and
\eqref{eq:Z(t)}), we can conclude that
\begin{align*}
    \lim_{t \to \infty} \frac{f(t, \tau)}{t }= h \alpha' c^2\int_0^\infty d\mathbb{P}'(L\leq u)
    = h \alpha' c^2 v \int_0^\infty e^{-2\alpha u}d\mathbb{P}(L\leq u) \quad \text{uniformly for $\tau \geq 0$.} 
\end{align*}

Moreover, if we consider the first of the integrals in~\eqref{eq:f_1} and apply the Cauchy-Schwartz inequality, we obtain
\begin{align*}
    \int_0^t \frac{\mathbb{E}(G(t-u))}{e^{\alpha(t-u)}}&\frac{\mathbb{E}(Z(t+\tau-u))}{e^{\alpha(t+\tau -u)}}d\mathbb{P}'(L\leq u) \notag \\
    \leq \left(\int_0^t \left|\frac{\mathbb{E}(G(t-u))}{e^{\alpha(t-u)}}\right|^2d\mathbb{P}'(L\leq u)\right)^{1/2}  & \left(\int_0^{t+\tau} \left|\frac{\mathbb{E}(Z(t+\tau-u))}{e^{\alpha(t+\tau -u)}}\right|^2d\mathbb{P}'(L\leq u)\right)^{1/2}
     =:f_1(t)f_2(t+\tau),
\end{align*}
with $f_1(t)$ and $f_2(t)$ satisfying the hypotheses of Theorem~\ref{thm:renewal theorem}. As the same reasoning holds for the other integrals in~\eqref{eq:f_1} (for the last integral we use Cauchy-Schwartz inequality twice), thanks to Theorem~\ref{thm:renewal theorem}, with $p=1$ and $q=0$, and Corollary~\ref{cor:1} we obtain 
\begin{equation*}
    \lim_{t\to \infty}\frac{K(t)}{t}=\frac{h \alpha' c^2 v \int_0^\infty e^{-2\alpha u}d\mathbb{P}(L\leq u)}{1- h\int_0^\infty e^{-2\alpha u}d\mathbb{P}(L\leq u)}\quad \text{ uniformly for $\tau \geq 0$.}
\end{equation*}
Recalling the definition of $k$, $\WGtime{t}$, and $\WZtime{t+\tau}$
at~\eqref{eq:k} and~\eqref{eq:W_and_V}, we have completed the proof
of the first inequality in~\eqref{eq:V_tW_t}.
\end{proof}

We now study the covariance between the total generation process at two distinct
times, for which we will need to use Lemma~\ref{lem:Z(t+tau)G(t)}.

\begin{lemma}[Limit behaviour of the covariance of $\WGtime{t}$ and $\WGtime{t+\tau}$]\label{lem:G(t+tau)G(t)}
Using the previous notation, we have
\begin{equation*}
    \lim_{t \to \infty}\mathbb{E}(\WGtime{t+\tau}\WGtime{t})=  k \qquad \text{ uniformly for $\tau \geq 0$, }
\end{equation*}
where $k$ is defined in~\eqref{eq:k}.
\end{lemma}

\begin{proof}
The proof is similar to that in Lemma~\ref{lem:Z(t+tau)G(t)}, so
some details are omitted.

If we take the derivative of equation~\eqref{eq:probability generating function} first for $s_1$, secondly for $s_2$, and then evaluate it at $(1,1,1,1,t,\tau)$, we obtain
\begin{align} \label{eq:media quadratica}
    \mathbb{E}(G(t+ \tau)G(t))=& v \int_0^t \mathbb{E}(G(t+\tau-u))\mathbb{E}(G(t-u))d\mathbb{P}(L\leq u) \notag\\
    & +v\int_0^t \Big[\mathbb{E}(Z(t+ \tau-u))\mathbb{E}(Z(t-u)) + \mathbb{E}(G(t+\tau-u))\mathbb{E}(Z(t-u)) \notag\\
    &+ \mathbb{E}(Z(t+ \tau-u))\mathbb{E}(G(t-u)) \Big]d\mathbb{P}(L\leq u) \notag \\
    & +h\int_0^t \Big[\mathbb{E}\big(G(t+\tau-u)Z(t-u)\big)  + \mathbb{E}\big(Z(t+ \tau-u)G(t-u)\big) \notag\\ 
    &+ \mathbb{E}\big(Z(t+ \tau-u)Z(t-u)\big)\Big]d\mathbb{P}(L\leq u) \notag\\
    & +h\int_0^t \mathbb{E}(G(t+\tau-u)G(t-u))d\mathbb{P}(L\leq u). 
\end{align}
Multiplying both sides of this equation by $e^{-\alpha t}e^{-\alpha (t+\tau)}$ and denoting 
\begin{align*}
    K(t,\tau) :=\frac{\mathbb{E}(G(t+ \tau)G(t))}{e^{\alpha t}e^{\alpha (t+\tau)}}, \quad
    d\mathbb{\overline{P}}(L\leq u):= h e^{- 2 \alpha u} d\mathbb{P}(L\leq u), \quad
    d\mathbb{P}'(L\leq u):=v e^{-2\alpha u}d\mathbb{P}(L\leq u),
\end{align*}
\begin{align}
    f(t, \tau):= & \int_0^t \frac{\mathbb{E}(G(t+\tau-u))}{e^{\alpha(t+\tau -u)}}\frac{\mathbb{E}(G(t-u))}{e^{\alpha(t-u)}}d\mathbb{P}'(L\leq u) \notag \\
    &+ \int_0^t \Big[\frac{\mathbb{E}(Z(t+ \tau-u))}{e^{\alpha(t+\tau -u)}}\frac{\mathbb{E}(Z(t-u))}{e^{\alpha(t-u)}} + \frac{\mathbb{E}(G(t+\tau-u))}{e^{\alpha(t+\tau -u)}}\frac{\mathbb{E}(Z(t-u))}{e^{\alpha(t-u)}} \notag\\
    &+ \frac{\mathbb{E}(Z(t+ \tau-u))}{e^{\alpha(t+\tau -u)}}\frac{\mathbb{E}(G(t-u))}{e^{\alpha(t-u)}} \Big]d\mathbb{P}'(L\leq u) \notag \\
    & +\int_0^t \Big[\frac{\mathbb{E}\big(G(t+\tau-u)Z(t-u)\big)}{e^{\alpha(t+\tau -u)}e^{\alpha(t-u)}}  +\frac{\mathbb{E}\big(Z(t+ \tau-u)G(t-u)\big)}{e^{\alpha(t+\tau -u)}e^{\alpha(t-u)}}  \notag\\ 
    &+ \frac{\mathbb{E}\big(Z(t+ \tau-u)Z(t-u)\big)}{e^{\alpha(t+\tau -u)}e^{\alpha(t-u)}}\Big]d\mathbb{\overline{P}}(L\leq u), \label{eq:f}
\end{align}
 we have that
\begin{equation}\label{eq:K(t,tau)}
   K(t,\tau)=f(t,\tau) + \int_0^t K(t-u, \tau) d\mathbb{\overline{P}}(L\leq u). 
\end{equation}

As already observed in~\eqref{eq:P_difective}, $\mathbb{\overline{P}}$ is a defective measure. 
In order to conclude the proof, we would like to apply Theorem~\ref{thm:renewal theorem} to~\eqref{eq:K(t,tau)}, and so we need to prove that the hypotheses on $f(t,\tau)$ are verified.
This will be easier by proving a weaker version of Lemma~\ref{lem:G(t+tau)G(t)} which states that $\lim_{t\to \infty}\mathbb{E}(G(t)^2)/[t^2e^{2\alpha t}] = (h \alpha' c)^2 k$. This result, that now we prove, is obtained applying the first part of Theorem~\ref{thm:renewal theorem} to~\eqref{eq:K(t,tau)}, when $\tau=0$.

For $\tau=0$,  we have that $K(t,0)=\mathbb{E}(G(t)^2)/e^{2\alpha t}$ and

\begin{align}\label{eq:f(tau,0)} 
    f(t,0)=& \int_0^t \left[\frac{\mathbb{E}(G(t-u))^2}{e^{2\alpha (t-u)}} +\frac{\mathbb{E}(Z(t-u))^2}{e^{2\alpha (t-u)}} +2 \frac{\mathbb{E}(G(t-u))}{e^{\alpha (t-u)}}\frac{\mathbb{E}(Z(t-u))}{e^{\alpha (t-u)}}\right] d\mathbb{P}'(L\leq u) \notag \\
    & +\int_0^t \left[2\frac{\mathbb{E}\big(G(t-u)Z(t-u)\big)}{e^{2\alpha (t-u)}} + \frac{\mathbb{E}(Z(t-u)^2)}{e^{2\alpha (t-u)}}\right]d\mathbb{\overline{P}}(L\leq u).
\end{align}
Notice that all five terms inside the integrals in~\eqref{eq:f(tau,0)} are
locally bounded in $t$ (the moments and the covariance of $Z(t)$
and $G(t)$ are locally bounded solutions of integral equations of
the type~\eqref{eq:K(t,tau)}, see \cite[Theorem 2]{weber2016inferring}) and, divided by $t^2$, they converge to
constants. So, we can use Lemma~\ref{lem:convergenza
uniforme} with $p=2$ and $q=0$, obtaining

\begin{equation}\label{eq:f(t,0)/t^2}
    \lim_{t\to \infty}\frac{f(t,0)}{t^2}=(h \alpha' c)^2\int_0^\infty d\mathbb{P}'(L\leq u)=(h \alpha' c)^2 v \int_0^\infty e^{-2\alpha u}d\mathbb{P}(L\leq u).
\end{equation}
As $f(t,0)$ is locally bounded in $t$ (it is finite sum of convolutions of locally bounded functions), equation \eqref{eq:f(t,0)/t^2} allows us to apply Theorem~\ref{thm:renewal theorem} obtaining
\begin{equation}\label{eq:E(G^2)}
    \lim_{t \to \infty} \frac{K(t,0)}{t^2}=\lim_{t \to \infty} \frac{\mathbbm{E}(G(t)^2)}{t^2 e^{2\alpha t}}= \frac{(h \alpha' c)^2 v \int_0^\infty e^{-2\alpha u}d\mathbb{P}(L\leq u)}{1- h\int_0^\infty e^{- 2 \alpha u} d\mathbb{P}(L\leq u)}=(h \alpha' c)^2 k.
\end{equation}

Let's go back to the proof of Lemma~\ref{lem:G(t+tau)G(t)} and see
that $f(t,\tau)$ satisfies the hypotheses of Theorem~\ref{thm:renewal
theorem}. In~\eqref{eq:f}, each of the seven integrands, when divided
by $t(t + \tau)$, converges to a constant  when
$t \to \infty$, uniformly for $\tau \geq 0$ (see~\eqref{eq:E(G(t))},\eqref{eq:Z(t)},\eqref{eq:V_tW_t},
and~\cite[pg. 145]{harris1963theory}). Furthermore, each of these
integrands is dominated by the product of two locally bounded functions,
one depending from $t$ and another one depending from $t+\tau$ (use
the Cauchy-Schwartz inequality for the last three integrands to see
it). As these functions satisfy the hypotheses of
Lemma~\ref{lem:convergenza uniforme}
(see~\eqref{eq:E(G(t))},\eqref{eq:Z(t)}, and~\eqref{eq:E(G^2)}),
we can conclude that 
\begin{align*}
    \lim_{t \to \infty} \frac{f(t, \tau)}{t(t +\tau)}= (h \alpha'
    c)^2\int_0^\infty d\mathbb{P}'(L\leq u) = (h \alpha' c)^2 v
    \int_0^\infty e^{-2\alpha u}d\mathbb{P}(L\leq u) \quad
    \text{uniformly for $\tau\geq 0$.}
\end{align*}

Moreover, using the Cauchy-Schwartz inequality (for the last three
integrals we have to use it twice), each of the integrals in~\eqref{eq:f}
are dominated by the product of two functions, one depending on $t$
and the other one on $t+\tau$, which satisfy the hypotheses of
Theorem~\ref{thm:renewal theorem}. So, Corollary~\ref{cor:1} implies
\begin{equation*}
    \lim_{t\to \infty}\frac{K(t)}{t(t+\tau)}=\frac{(h \alpha' c)^2 v \int_0^\infty e^{-2\alpha u}d\mathbb{P}(L\leq u)}{1- h\int_0^\infty e^{-2\alpha u}d\mathbb{P}(L\leq u)}\quad \text{ uniformly for $\tau \geq 0$.}
\end{equation*}
The definitions of $k$ and $\WGtime{t}$ at~\eqref{eq:k} and~\eqref{eq:W_and_V}, respectively, allow to conclude the proof.
\end{proof}

An immediate consequence of this lemma is the following.

\begin{corollary}[$\WGtime{t}$ is a Cauchy sequence in $L^2$.]\label{cor:Cauchy} Using the previous notation, we have
\begin{equation*}
    \lim_{t \to \infty} \mathbb{E}((\WGtime{t+\tau}-\WGtime{t})^2) \to 0 \qquad \text{uniformly for $\tau \geq 0.$}
\end{equation*}
\end{corollary}

\begin{proof}
From Lemma~\ref{lem:G(t+tau)G(t)}, uniformly for $\tau \geq 0$, we have that 
\begin{align*}
   \lim_{t \to \infty} \mathbb{E}((\WGtime{t+\tau}-\WGtime{t})^2)=  \lim_{t \to \infty} \Big[\mathbb{E}(\WGtime{t+\tau}^2) + \mathbb{E}(\WGtime{t}^2) -2\mathbb{E}(\WGtime{t+\tau}\WGtime{t})\Big] 
   = k +k - 2k = 0
\end{align*}
\end{proof}

We have just proved that $\WGtime{t}$ is a Cauchy sequence in $L^2$, i.e. for every $ \epsilon > 0$ there exists a $t_\epsilon> 0$ s.t. for every $ t>t_\epsilon$ and $\tau \geq0 $ we have $\mathbb{E}((\WGtime{t+\tau}-\WGtime{t})^2)<\epsilon.$ Thanks to the completeness of the $L^2$ space, we can now easily prove Theorem~\ref{th:L2 convergence}.

\begin{theorem}[Mean square convergence of $G(t)$]\label{th:L2 convergence}
There exists a non-negative random variable $\WG \in L^2$ such that 
\begin{equation*}
    \lim_{t \to \infty} \mathbb{E}((\WGtime{t}-\WG)^2)=0,
  \end{equation*}
with $\mathbb{E}(\WG)=1$ and $\Var(\WG)=k-1= [(v+h) \int_0^\infty e^{- 2 \alpha u} d\mathbb{P}(L\leq u) -1]/[1 - h \int_0^\infty e^{- 2 \alpha u} d\mathbb{P}(L\leq u) ] >0$.
\end{theorem}

\begin{proof}
The existence of a such $\WG$ follows from Corollary~\ref{cor:Cauchy},
the fact that the  $L^2$ space is complete, and that $\WGtime{t}$
satisfies the Cauchy criterion for convergence in $L^2$.
Using~\eqref{eq:E(G(t))} and the fact that $L^2 \subset L^1$,  we
know that $\mathbb{E}(\WG)=\lim_{t \to \infty}\mathbb{E}(\WGtime{t})=1$,
so it remains only to compute the variance. From the $L^2$ convergence
we have that $\mathbb{E}(\WG^2)=\lim_{t \to \infty}
\mathbb{E}(\WGtime{t}^2)$. Then,
\begin{align}\label{eq:Var_V}
    \Var(\WG)= \mathbb{E}(\WG^2) - \mathbb{E}(\WG)^2 =\lim_{t \to \infty} \mathbb{E}(\WGtime{t}^2) -1 =k-1
    \overset{\eqref{eq:k}}{=}  \frac{(v+h) \int_0^\infty e^{- 2 \alpha u} d\mathbb{P}(L\leq u) -1}{1 - h \int_0^\infty e^{- 2 \alpha u} d\mathbb{P}(L\leq u) }.
\end{align}
The positivity of~\eqref{eq:Var_V} follows from the same argument used by Harris in~\cite[pg. 146]{harris1963theory}.
Indeed, there he proved that the process $\WZtime{t}$ converges a.s. to a random variable $\WZ$ with the same mean and variance as  $\WG$.
\end{proof}

Theorem~\ref{th:L2 convergence} gives us the mean square convergence
of $\WGtime{t}$, which implies also the convergence in probability
and in mean. In Section~\ref{ch:almost sure convergence} we will
see that the convergence is also true with probability one.

\subsection{Functional equation for the MGF of $(\WG, \WZ)$} \label{ch:MGF W}

A surprising consequence of Theorem~\ref{th:L2 convergence}
and~\cite[Theorem 19.1]{harris1963theory} is that the processes
$\WG$ and $\WZ$ share the same mean and variance. In this section,
using the Moment Generating Function (MGF) of the pair $(\WG,\WZ)$,
we prove that these two variables are actually almost surely equal.
That is, on a path-by-path basis, the prefactor for the normalised
population size and for the normalised total generation is the same
with probability one.

\begin{theorem}[$Z(t)$ and $G(t)$ have same randomness in their dominant terms]\label{th:V=W}
Given 
\begin{equation*}
\frac{G(t)}{c h \alpha'  t e^{\alpha t}} 
=\WGtime{t} \overset{L^2}{\rightarrow}\WG 
\text{ and }
\frac{Z(t)}{c e^{\alpha t}}
=\WZtime{t} \overset{a.s.}{\rightarrow}
\WZ  
\end{equation*}
we have that 
\begin{equation*}
    \WG=\WZ \quad \text{ a.s.}
\end{equation*}
\end{theorem}

\begin{proof}
The proof is divided in two parts: first, we prove that $\WG$ and $\WZ$ are equally distributed, then that they coincide with probability one.

Theorem~\ref{th:L2 convergence}, together with~\eqref{eq:Z(t)}, imply that  $(\WGtime{t}, \WZtime{t}) \xrightarrow[]{D} (\WG,\WZ)$ in distribution. So, we can characterise the distribution of the pair $(\WG, \WZ)$ studying the MGF of  $(\WGtime{t}, \WZtime{t})$ when $t \to \infty$.
 
Proposition~\ref{prop:F} gives us an equation solved by the PGF of the vector $(G(t), G(t+ \tau), Z(t), Z(t+\tau))$. Evaluating this equation in $(s_1,1,r_1,1,t,0)$, we obtain the following expression solved by the PGF $F(s_1,r_1,t)$ of $(G(t),Z(t))$

\begin{equation} \label{eq:F(s,r,t)}
    F(s_1,r_1,t)=\mathbb{E}(s_1^{G(t)}r_1^{Z(t)}) = r_1\mathbb{P}(L>t) + \int_0^t\rho_N\Big(F(s_1,s_1r_1,t-u)\Big) d\mathbb{P}(L\leq u).
\end{equation}

Replacing $s_1$ with $\exp(-s/[h c \alpha' t e^{\alpha t}])$ and $r_1$ with $\exp(-r/[c e^{\alpha t}])$, for $s,r\geq 0$,  we obtain an expression solved by the MGF $\phi(s,r,t)$ of $(\WGtime{t},\WZtime{t})$:
\begin{align*} 
    \phi(s,r,t)&= \mathbb{E}\Big(e^{-\frac{s G(t)}{hc \alpha' t e^{\alpha t}}}e^{-\frac{r Z(t)}{c e^{\alpha t}}}\Big) \notag\\
    &=e^{-\frac{r}{c e^{\alpha t}}}\mathbb{P}(L>t) + \int_0^t \rho_N \Big(\mathbb{E}\Big(e^{-\frac{(t-u)s e^{-\alpha u}}{t}\WGtime{t-u}}e^{-\frac{(s + h r\alpha' t) e^{-\alpha u}}{h \alpha' t}\WZtime{t-u}}\Big)\Big)d\mathbb{P}(L\leq u) \notag \\
    &= e^{-\frac{r}{c e^{\alpha t}}}\mathbb{P}(L>t) + \int_0^t \rho_N \Big(\phi\Big(\frac{(t-u)}{t}s e^{-\alpha u},\frac{(s + h r\alpha' t) }{h 
\alpha' t}e^{-\alpha u}, t-u\Big)\Big) d\mathbb{P}(L\leq u).
\end{align*}

 Taking the limit for $t\to \infty$ of $\phi(s,r,t)$, we obtain that $\mathbb{E}(\exp(-s\WG) \exp(-r \WZ))$ solves the integral equation 
\begin{equation}\label{eq:MGF coppia limite}
    \phi(s,r)=\int_0^{\infty} \rho_N \Big(\phi\Big(s e^{-\alpha u},re^{-\alpha u}\Big) \Big)d\mathbb{P}(L\leq u) \quad  s,r\geq 0 .
\end{equation}
This means that if we consider $r=0$, the function $\mathbb{E}[\exp(-s \WG) ]$, that represents the MGF of $\WG$, solves the integral equation
\begin{equation}\label{eq: MGF W}
    \psi(s)=\int_0^{\infty} \rho_N \Big(\psi\big(s e^{-\alpha u}\big)\Big) d\mathbb{P}(L\leq u), \quad {s\geq 0}
\end{equation}
with $\psi(0)=0$ and $\psi'(0)=-1$. 
The uniqueness of the solution of this problem~\cite[Theorem 4.1]{levinson1960limiting} 
and the fact that the MGF of the variable
$\WZ$ solves \eqref{eq: MGF W} too~\cite[pg. 146]{harris1963theory},
give us that the MGFs of $\WZ$ and $\WG$ coincide for $s\geq 0$.
Using a result proved by Mukherjea et al.~\cite[Theorem 2]{mukherjea2006note}, we can
conclude that $\WZ$ is equal in distribution to $\WG$.

Now, if we consider $r=s$ in~\eqref{eq:MGF coppia limite}, we can see that the function $\mathbb{E}[\exp(-s(\WG+\WZ)) ]$, that represents the MGF of $\WG+\WZ$, solves \eqref{eq: MGF W} but with the initial conditions $\psi(0)=0$ and $\psi'(0)=-2$. 
Another solution of \eqref{eq: MGF W} with the same initial conditions is given by $2\WZ$. Also in this case, the uniqueness of the solution and \cite[Theorem 2]{mukherjea2006note} allows us to conclude that $2\WZ\overset{D}{=} \WZ+ \WG$.

These last two results give us that $\WZ\overset{a.s.}{=}\WG$. In fact, $\WZ\overset{D}{=}\WG$ implies that $\Var(\WZ) = \Var(\WG)$, and 
\begin{align*}
    2\WZ\overset{D}{=} \WZ+ \WG \quad \Longrightarrow& \quad 4\Var(\WZ)= \Var(\WZ) + \Var(\WG) + 2\Cov(\WZ,\WG) \notag \\
    \Longrightarrow \quad  \Var(\WZ)= \Cov(\WZ,\WG) 
    &\quad \Longrightarrow \quad  \text{Corr}_{\WZ,\WG}:=\frac{\Cov(\WZ,\WG)}{\sqrt{\Var(\WZ)}\sqrt{\Var(\WG)}}=1, 
\end{align*}
where in the last inequality we have used the definition of Pearson's correlation coefficient. The correlation coefficient equal to 1 implies that $\WG=a\WZ +b$ a.s., for $a \geq 0$, $b \in \mathbb{R}$~\cite[Theorem 4.5.7]{casella2002statistical}. From $\WZ\overset{D}{=}\WG$, we obtain $a=1$ and $b=0$, i.e. $\WZ\overset{a.s.}{=}\WG$. This conclude the proof.
\end{proof}

Thus, from Theorem~\ref{th:V=W}, $\WZ$ can be used in lieu of $\WG$ from here on.

\subsection{Almost sure convergence of $G(t)$}\label{ch:almost sure convergence}

We have gathered the results needed to establish one of the significant
results of the article: the almost sure convergence of a normalised
version of the process $\{G(t)\}$.  In order to prove that, we will assume something concerning
the speed of convergence of $\WGtime{t}$ to $\WZ$ as $L^2$ functions.
This assumption is equivalent to the one made by Harris in~\cite[Chapter
VI, Theorem 21.1]{harris1963theory} concerning the size of the
population, which - for the population size - was later established
by Jagers~\cite{jagers1969renewal} to be unnecessary.

\begin{theorem}[Almost sure convergence of $G(t)$] \label{th:G(t) almost sure}
If $\int_0^{\infty} \mathbb{E}((\WGtime{t}-\WZ)^2) dt < \infty$, we have that 
\begin{equation*}
\frac{G(t)}{h \alpha' c t e^{\alpha t}}=\WGtime{t}
     \xrightarrow[t \to \infty]{a.s} \WZ.  
\end{equation*}
\end{theorem}

\begin{proof}
We start with the additional hypothesis $p_0=\mathbbm{P}(N=0)=0$
in order to have $G(t)$ as a finite, non-decreasing step function
of $t$. Using Fubini's theorem on $\int_0^{\infty}
\mathbb{E}((\WGtime{t}-\WZ)^2) dt < \infty$, we obtain that
$\mathbb{P}(\int_0^{\infty}(\WGtime{t}-\WZ)^2 dt < \infty)=1$.
Since $G(t)$ is non-decreasing in t, we have
 \begin{equation}\label{eq:disuguaglianza W tilde tau}
     \WGtime{t+\tau}=\frac{G(t+ \tau)}{h c \alpha'  (t+\tau) e^{\alpha (t+\tau)}} \geq  \frac{t}{(t+ \tau)e^{\alpha \tau}} \frac{G(t)}{h c \alpha'  t e^{\alpha t}}= \frac{t}{(t+ \tau)e^{\alpha \tau}} \WGtime{t},
 \end{equation}
where the inequalities are true for every realisation of the random variables.

Let's suppose that $\WGtime{t} \xrightarrow[t \to \infty]{a.s.} \WZ$ is not true .
If $(\Omega, \mathcal{B}(\Omega), \mathbb{P})$ is the probability space where $\WGtime{t}$ and $\WZtime{t}$ are defined, then there exists a set $A\subseteq \{\omega \in \Omega| \lim_{t\to \infty} \WGtime{t}(\omega)\neq \WZ(\omega)\}$ that is measurable and such that $\mathbb{P}(A)>0$.
Since $\WZ>0$ a.s. \cite[Remark 1, Section 20]{harris1963theory}, we can also suppose that $\WZ(\omega)>0$ for every $\omega \in A$.

For every $\omega \in A$ we have that at least one between $\limsup_{t \to \infty}\WGtime{t}(\omega) > \WZ(\omega)$ and $\liminf_{t \to \infty}\WGtime{t}(\omega) < \WZ(\omega)$ is true. We will see that in both cases we will have $\int_0^{\infty} (\WGtime{t}(\omega)-\WZ(\omega))^2 dt = + \infty$, leading to the contradiction  $\mathbb{E}(\int_0^{\infty}(\WGtime{t}-\WZ)^2 dt) = + \infty$.

Let us start fixing $\omega \in A$ and assuming $\limsup_{t \to \infty}\WGtime{t}(\omega) > \WZ(\omega)$.
This implies that there exist a $ \delta >0$ and a sequence $(t_i)_{i \in \mathbb{N}}$, with $\lim_{i \to \infty} t_i=\infty$, such that $\WGtime{t_i}(\omega)>(1+\delta) \WZ(\omega)$, $i \in \mathbb{N}$. If we consider $0<\epsilon<\delta$, without loss of generality we can choose this sequence such that 
\begin{equation*}
t_{i+1} - t_i > \frac{(\delta-\epsilon) t_i}{1+ \epsilon + \alpha t_i(1+\delta)}:=b_i.    
\end{equation*}
Note that $\delta, \epsilon$, and $t_i$ depend on $\omega$ and that $(b_i)_{i\in \mathbb{N}}$ and $(t_i)_{i\in \mathbb{N}}$ are monotonically increasing.

Using~\eqref{eq:disuguaglianza W tilde tau} and the relation $e^{- \alpha \tau} \geq 1 - \alpha \tau$, we obtain for every $i \in \mathbb{N}$
\begin{align}
   \WGtime{t_i+\tau}(\omega)&\overset{\eqref{eq:disuguaglianza W tilde tau}}{\geq} \frac{t_i}{(t_i+ \tau)e^{\alpha \tau}} \WGtime{t_i}(\omega) > \frac{t_i}{t_i+ \tau} (1 - \alpha \tau)(1+\delta)\WZ(\omega), &\tau \in (0, \infty) \notag \\
   &\geq (1+\epsilon) \WZ(\omega) & \tau \in (0, b_i) \notag \\
   &\geq (1+\epsilon) \WZ(\omega) &\tau \in (0, b_1), \label{eq:(1+epsilon)}
\end{align}
where we have used the fact that the function $t_i (t_i+ \tau)^{-1} (1 - \alpha \tau)(1+\delta)$ is decreasing in $\tau$, that for $\tau=b_i$ it is equal to $(1+ \epsilon)$, and that $(b_i)_{i \in \mathbb{N}}$ is an increasing sequence.

Hence, using~\eqref{eq:(1+epsilon)}, we have for every $i$ that
\begin{align*}
    &\int_{t_i}^{t_{i+1}}(\WGtime{t}(\omega)-\WZ(\omega))^2 dt \geq  \int_{t_i}^{t_i + b_1}(\WGtime{t}(\omega)-\WZ(\omega))^2 dt \\
    = &\int_{0}^{b_1}(\WGtime{t_i+\tau}(\omega)-\WZ(\omega))^2 d\tau
    \geq(\epsilon \WZ(\omega))^2 b_1 >0.
\end{align*}
This allows us to say that $\int_0^{\infty} (\WGtime{t}(\omega)-\WZ(\omega))^2 dt = + \infty$.

Same conclusion can be obtained assuming $\liminf_{t \to \infty} \WGtime{t}(\omega) < \WZ(\omega)$. Indeed,
for the definition of $\liminf$ we have that there exist $ \delta \in (0,1)$ and a sequence $(t_i)_{i \in \mathbb{N}}$, with $t_i>1$ and $\lim_{i \to \infty} t_i=\infty$, such that $\WGtime{t_i}< (1- \delta) \WZ$. We can also pretend that $t_{i+1} - t_i > a >0$, where $a$ is chosen in order to satisfy the following inequalities for $i$ big enough
\begin{align*}
    0< \WGtime{t_i - \tau} & \overset{\eqref{eq:disuguaglianza W tilde tau}}{\leq}  \frac{t_i}{t_i - \tau} e^{\alpha \tau} \WGtime{t_i} < (1- \delta) \frac{t_i}{t_i - \tau} e^{\alpha \tau} \WZ &\tau \in (0, t_1)\\
    &\leq (1 - \epsilon)\WZ &\tau \in (0, a),
\end{align*}
 where $\epsilon$ is a constant s.t. $0 <\epsilon < \delta$.
The existence of such $a$ is consequence of the fact that $\psi(t,\tau):=(1- \delta) e^{\alpha \tau} t/(t - \tau) $, as long as $\tau<t$, is increasing in $\tau$ and decreasing in $t$. Indeed, this implies that there exists $a>0$ s.t. for $\tau \in [0,a]$ $(1-\delta)= \psi(1,0)\leq \psi(1,\tau)\leq(1-\epsilon)$, from which we can conclude that for $\tau \in [0,a]$ and $t \geq 1$, we have $(1-\delta)= \psi(t,0)\leq \psi(t,\tau)\leq(1-\epsilon)$ .
 
 Then, we have
  \begin{align*}
    &\int_{t_{i-1}}^{t_{i}}(\WZ(\omega)-\WGtime{t}(\omega))^2 dt \geq  \int_{t_i-a}^{t_i}(\WZ(\omega)-\WGtime{t}(\omega))^2 dt \\
    \geq & \int_{0}^{a}(\WZ(\omega)-\WGtime{t_i - \tau}(\omega))^2 d\tau \geq
    (\epsilon \WZ(\omega))^2 a.
\end{align*}
As before, this implies that $\int_0^{\infty} (\WGtime{t}(\omega)-\WZ(\omega))^2 dt = + \infty$.

So, for every $\omega \in A$ we have $\int_0^{\infty} (\WGtime{t}(\omega)-\WZ(\omega))^2 dt = + \infty$ and, because $\mathbb{P}(A)>0$, we have $\mathbb{E}(\int_0^{\infty}(\WGtime{t}-\WZ)^2 dt) = + \infty$. This contradicts the hypothesis of the theorem and so we have proved that $\lim_{t \to \infty} \WGtime{t}=\WZ$ with probability 1 under the condition  $p_0=0$.

When $p_0 \neq 0$, we can observe that $G(t)=G_B(t) - G_D(t)$, where
$G_B(t)$ and $G_D(t)$ are the sum of the generation of the cells
born and dead  before or at time $t$, respectively.  Also for these
processes we can find integral equations for the probability
generating function similar to the one found for $G(t)$ and repeat
all the previous steps. Thanks to the monotonicity of $G_B(t)$ and
$G_D(t)$, this time we don't need the assumption $p_0=0$, obtaining
the almost sure convergence of $G_B(t)/n_1t e^{\alpha t}$ and
$G_D(t)/n_2t e^{\alpha t}$ to the random variables $\WZ_B$ and
$\WZ_D$ respectively, where $n_1,n_2$ are positive constants. This
allows us to conclude that $\WGtime{t}$ converges to $\WZ_B + \WZ_D$.
\end{proof}

Having established the almost sure result for the limiting behaviour
of the total generation process $G(t)$, we are in a position to the
final deduction of the section that leads to equation \eqref{eq:main_result}.
Thanks to equation~\eqref{eq:Z(t)}, Theorem~\ref{th:G(t) almost
sure}, and the Continuous Mapping Theorem, we have the following
corollary.

\begin{corollary}[Almost sure average generation inference]\label{cor:main}
If $\E(N^2)<\infty$,   $\liminf_{t\to\infty} \Zp(t)>0$, and $\int_0^{\infty} \mathbb{E}((\WGtime{t}-\WZ)^2) dt < \infty$, we have that
\begin{equation*}
    \lim_{t\to \infty}\frac{G(t)}{t Z(t)} = h \alpha' = -\lim_{p\to 0} \lim_{t \to \infty} \frac{1}{pt}\log\left(\frac{Z^+(t)}{Z(t)}\right) \qquad \text{almost surely}.
\end{equation*}
\end{corollary} 

Thus the average estimation scheme first proposed \cite{weber2016inferring}
that is based on a one-way, heritable change in a neutral label is
almost surely correct on a path-by-path basis for a Bellman-Harris
branching process.

\section{A two-type Bellman-Harris process subject to one-way differentiation}

In addition to division and death, cells often undergo changes in
cell-type. For example, many tissues are formed through progressive
stages of proliferation and change in cell-type, called cellular
differentiation, from stem
cells~\cite{kondo1997identification,akashi2000clonogenic}, while
cancer cells arise as mutants with abherent DNA from healthy
cells~\cite{mendelsohn2015molecular,hong2010holland}. Changes in
cell-type are often accompanied by changes in population kinetics,
e.g. \cite{Akinduro18}, and to better understand these differentiation
processes it can be desirable to obtain information on the average
generation of each population as they are often reported as being
division-linked
\cite{hodgkin1996b,deenick1999switching,Duffy12,pauklin2013cell}.

As a basic model of changes in cell type, in the present section
we extend the previous results to a two-type Bellman-Harris branching
process subject to one-way differentiation, a model first considered
in \cite{jagers1969proportions} where cells of one type can give
rise to another but not vice-versa. These results significantly
extend the remit and utility of the inference of average generation
by random delabelling. In particular, if the initial cell is equipped
with a neutral label that is heritably lost with a fixed probability
per division, we prove that the average generation of each cell-type
can be inferred from knowledge of that probability and the proportion
of label positive cells. Before stating the results, we introduce
notation that is consistent with that used in Section~\ref{ch:assumption}
and with that employed in~\cite{jagers1969proportions}, where sample
path results for the population size were first established
in this two-type setting.

As in Fig. \ref{fig:cell_population}, consider a cell population
whose members are from two types, type-1 and type-2. Each cell lives
a random type-dependent lifetime $L_i$, $i\in\{1,2\}$, after which
it dies or divides generating $N_i$ offspring. We assume $L_i$ and
$N_i$ are independent for each cell, and amongst all cells.
Furthermore, we suppose that only type-1 cells can generate cells
of both types, i.e. $N_1$ takes values in $\mathbbm{N}^2$ and has
PGF $\rho_1$, whereas the offspring of type-2 cells are exclusively
type-2 cells, so that $N_2$ takes value in $\mathbbm{N}$ and has
PGF $\rho_2$. We denote by $h_i:=(\partial/\partial x_i) \rho_1(1,1)$
the average number offspring of type-$i$ generated from a type-1
cell and, with $\mu:= d/dx \rho_2(1)$, the average number of offspring
obtained from a type-2 cell. As in the single-type case, we suppose
that $h_1$ and $\mu$ are greater than $1$ so that both populations
are super-critical.

We assign a generation to each cell, the integer that records how
many divisions led to that cell (Fig. \ref{fig:cell_population}).
We define cells a time zero as being in generation zero. Furthermore,
we suppose the cells in the initial population are equipped with a
neutral label (i.e. one that does not influence population dynamics)
that, independently for each cell, is heritably lost immediately
prior to a cell's division with probability $p$. For $i\in \{1,2\}$,
we denote by $Z_i(t)$ the total number of type-$i$ cells in the
population at time $t$, by $G_i(t)$ the total generation of type-$i$
cells at time $t$, and by $Z^+_i(t)$ the size of type-$i$ label-positive
at time $t$. To describe the growth rates of these processes,
we will need the Malthusian parameters, $\alpha_1$ and $\alpha_2$,
that are the solutions of the equations
\begin{equation}\label{eq:alpha_1-alpha_2}
    h_1 \E\left(e^{-\alpha_1 L_1 t}\right) = 1
\quad \text{and} \quad
    \mu \E\left(e^{-\alpha_2 L_2 t}\right) = 1.
\end{equation} 
The existence and the uniqueness of the solutions
of these equations are guaranteed by the hypotheses $h_1>1$ and
$\mu>1$. As in Section~\ref{ch:assumption}, we denote the derivatives
of the Malthus parameters as a function of the average offspring number by 
\begin{equation*}
\alpha'_1 = \frac{1}{h_1^2 \int_0^{+ \infty}t e^{- \alpha_1 t}d\mathbb{P}(L_1\leq u)}
\quad \text{and} \quad 
\alpha'_2 = \frac{1}{\mu^2 \int_0^{+ \infty}t e^{- \alpha_2 t}d\mathbb{P}(L_2\leq u)}.  
\end{equation*}

The population dynamics of type-1 cells are unaffected by type-2
cells and, treating differentiation as death, behave as a single
type process. If the starting population only has type-2 cells, the
system is again in the single type setting. Thus the interesting
setup is when the system is initiated with cells of type-1 and
queries are of the population size and average generation of
type-2 cells.

Let $\mathbbm{P}_i$ and $\mathbbm{E}_i$ denote the probability and
the expectation conditional on the population starting with a single
cell of type $i\in\{1,2\}$. The growth of the type-2 population size
given one initial type-1 cell, $Z_2(t)$ under $\P_1$, is studied
in~\cite{jagers1969proportions}. Those results can be immediately
applied to study $Z_2^+(t)$, given the first cell is type-1 and
label-positive. Analogous results for $G_2(t)$ can be obtained by repeating
the steps made in the single type case. In particular, adapting the
integral equation~\eqref{eq:probability generating function} to the
two-type problem, using Lemma~\ref{lem:convergenza uniforme} and
Theorem~\ref{thm:renewal theorem} we can establish the growth rates
of $\E_1(G_2(t)Z_2(t))$, $\E_1(G_2(t)^2)$,  $\E_1(Z_2(t)Z_2(t+\tau))$,
$\E_1(G_2(t)Z_2(t+\tau))$, $\E_1(G_2(t+\tau)Z_2(t))$, and
$\E_1(G_2(t)G_2(t+\tau))$. These results enable us to conclude the
mean square limit of $G_2(t)$ under $\P_1$. Stepping from the $L^2$
result to the almost sure one is achieved in the same way as
Theorem~\ref{th:G(t) almost sure}. As this line of reasoning is
essentially a replication of what is done in the single type case,
the details are omitted. From these, starting with one label-positive
type-1 cell, the in-expectation result relating the average generation
to the proportion of labelled cells follows immediately:
\begin{align*}
\lim_{t\to\infty} \frac{\E_1(G_2(t))}{t \E_1(Z_2(t))} = 
-\lim_{p\to0}\lim_{t\to\infty}\frac{1}{pt}\log\left(\frac{\E_1(Z^+_2(t))}{\E_1(Z_2(t))}\right). 
\end{align*}
This equation says that, on average, the average generation of the
type-2 population
can be determined from averages of the delabelling proportion. To
obtain stronger convergence results, one notes that a combination
of \cite[Theorems 19.1 and 21.1]{harris1963theory}, Theorem~\ref{th:L2
convergence}, and Theorem~\ref{th:G(t) almost sure} gives that
\begin{equation}\label{eq:Z_i_and_G_i_two_types}
    \lim_{t \to \infty} \frac{Z_i(t)}{c_ie^{\alpha_i t}} 
	\overset{L^2,a.s.}{=} \mathcal{Z}_i 
	\quad \text{and} \quad 
	\lim_{t \to \infty} \frac{G_i(t)}{d_i te^{\alpha_i t}} 
	\overset{L^2,a.s.}{=}  \mathcal{Z}_i \quad \text{under $\mathbbm{P}_i$},
\end{equation}
where
\begin{equation*}
    c_1=\frac{h_1-1}{h_1^2 \alpha_1 \int_0^\infty t e^{-\alpha_1 t} d\mathbbm{P}(L_1\leq t)}, \qquad 
    c_2=\frac{\mu-1}{\mu^2 \alpha_2 \int_0^\infty t e^{-\alpha_2 t} d\mathbbm{P}(L_2\leq t)},
\end{equation*}
$d_1=c_1 h_1\alpha_1' $, $d_2=c_2 \mu \alpha_2'$, and 
assuming $\int_0^{\infty} \mathbb{E}[(G_i(t)/(d_i te^{\alpha_i t}) -\mathcal{Z}_i)^2] dt < \infty$ for the almost sure results concerning $\{G_i(t)\}$ in~\eqref{eq:Z_i_and_G_i_two_types}. 
Moreover, from~\cite{weber2016inferring} we have also that, if $\lim_{t\to \infty}Z_i^+(t) > 0$ 
\begin{equation*}
    \lim_{p\to 0}\lim_{t\to \infty} -\frac{1}{pt} \log\left(\frac{Z_i^+(t)}{Z_i(t)}\right)\overset{a.s.}{=}
    \begin{cases}
    h_1\alpha'_1 & \text{if $i=1$}\\
    \mu\alpha'_2 & \text{if $i=2$}
    \end{cases}
  \quad \text{under $\mathbbm{P}_{i}$,}
\end{equation*}
where we supposed that the first cell is label positive.

We present two sets of results depending on whether $\alpha_1>\alpha_2$ or
vice versa. If $\alpha_1< \alpha_2$, which would model, for example,
the creation of cancer cells, the growth rate of the type-2 cells
is greater than the type-1 cells and their average generation is
determined by the derivative of the latter Malthus parameter.

\begin{proposition}[$\alpha_1< \alpha_2$]\label{prop:multitype_alpha1<alpha2}  
If $(\partial/\partial x_i x_j) \rho_1(1,1)$, for $1\leq i\leq j \leq 2$, and $(\partial/\partial x^2)\rho_2(1)$ are finite, we have that 
\begin{equation}\label{eq:Z_G_multitype_12}
    \lim_{t \to \infty} \frac{Z_2(t)}{c_{1,2}e^{\alpha_2 t}} \overset{L^2, a.s.}{=} \mathcal{W}, \quad \text{and} \quad \lim_{t \to \infty} \frac{G_2(t)}{d_{1,2} te^{\alpha_2 t}} \overset{L^2}{=}  \mathcal{W} \quad \text{under $\mathbbm{P}_1$},
\end{equation}
where
\begin{equation}\label{eq:c12,d12}
    c_{1,2}=\frac{h_2 c_2 \int_0^\infty e^{-\alpha_2 t}d\mathbbm{P}(L_1 \leq t)}{1-h_1\int_0^\infty e^{-\alpha_2 t}d\mathbbm{P}(L_1\leq t)}, \qquad 
    d_{1,2}=c_{1,2} \mu \alpha_2',
\end{equation}
and $\mathcal{W}$ is a non-negative random variable such that 
$\P_1(\mathcal{W}=0)=\P_1(\lim_{t\to \infty}Z_1(t) = 0, \lim_{t\to \infty}Z_{2}(t) = 0)$ and $\E_1(\mathcal{W})=1$.

If $\int_0^{\infty} \E_1[(G_2(t)/(d_{1,2} te^{\alpha_2 t})
-\mathcal{W})^2] dt < \infty$, the second limit in
\eqref{eq:Z_G_multitype_12} is also true almost surely.
Assuming the initial cell is of type-1, i.e. $Z_1^+(0)=1$ and
$Z_2(0)=G_1(0)=G_2(0)=0$, we have 
\begin{equation*}
    \lim_{t \to \infty}  \frac{G_2(t)}{tZ_2(t)}\overset{a.s.}{=}\mu \alpha_2' \overset{a.s.}{=} \lim_{p\to 0}\lim_{t\to \infty} -\frac{1}{pt} \log\left(\frac{Z_2^+(t)}{Z_2(t)}\right) \quad \text{if $\lim_{t\to \infty}Z_2^+(t) > 0$}.
\end{equation*}
\end{proposition}

If $\alpha_2 < \alpha_1$, as might occur with the production of terminally
differentiated cells, the growth rate of the type-1 cells is greater
than the type-2 cells and their average generation is determined
by the derivative of the former Malthus parameter. That is, in this
setting, so long as the type-1 population continues to exist, the
average generation of the type-2 cells is dominated by immigrants
from the type-1 population.
\begin{proposition}[$\alpha_2< \alpha_1$]\label{prop:multitype_alpha2<alpha1} 
If $(\partial/\partial x_i x_j) \rho_1(1,1)$, for $1\leq i\leq j \leq 2$, and $(\partial/\partial x^2)\rho_2(1)$ are finite, we have that 
\begin{equation}\label{eq:Z_G_multitype_21}
    \lim_{t \to \infty} \frac{Z_2(t)}{c_{2,1}e^{\alpha_1 t}} \overset{L^2,a.s.}{=} \mathcal{Z}_2 \quad \text{and} \quad \lim_{t \to \infty} \frac{G_2(t)}{d_{2,1} te^{\alpha_1 t}} \overset{L^2}{=}  \mathcal{Z}_2 \quad \text{under $\mathbbm{P}_1$},
\end{equation}
where
\begin{equation}\label{eq:c21,d21}
    c_{2,1}=\frac{h_2(1-\int_0^\infty e^{-\alpha_1 t}d\mathbbm{P}(L_2 \leq t))}{h_2^2 \alpha_1 (1-\mu \int_0^\infty e^{-\alpha_1 t}d\mathbbm{P}(L_2 \leq t))}, \qquad
    d_{2,1}=c_{2,1} h_1 \alpha_1',
\end{equation}
and $\mathcal{Z}_2$ random variable defined in~\eqref{eq:Z_i_and_G_i_two_types} with $\P_1(\mathcal{Z}_2=0)=\P_1(\lim_{t \to \infty} Z_1(t)=0)$ and $\E_1(\mathcal{Z}_2)=1$.

If $\int_0^{\infty} \E_1[(G_2(t)/(d_{2,1} te^{\alpha_2 t}) -\mathcal{Z}_2)^2] dt < \infty$, the second limit in \eqref{eq:Z_G_multitype_21} is also true almost surely.  Assuming the initial cell is of type-1, i.e. $Z_1^+(0)=1$ and
$Z_2(0)=G_1(0)=G_2(0)=0$, we have
\begin{equation*}
    \lim_{t \to \infty}  \frac{G_2(t)}{tZ_2(t)}\overset{a.s.}{=}h_1 \alpha_1' \overset{a.s.}{=} \lim_{p\to 0}\lim_{t\to \infty} -\frac{1}{pt} \log\left(\frac{Z_2^+(t)}{Z_2(t)}\right) \quad \text{if $\lim_{t\to \infty}Z_1^+(t) > 0$}.
\end{equation*}
\end{proposition}

\begin{figure}
\centering
\subfigure[normalised populations size ($\alpha_1<\alpha_2$)]{
\includegraphics[scale=0.4]{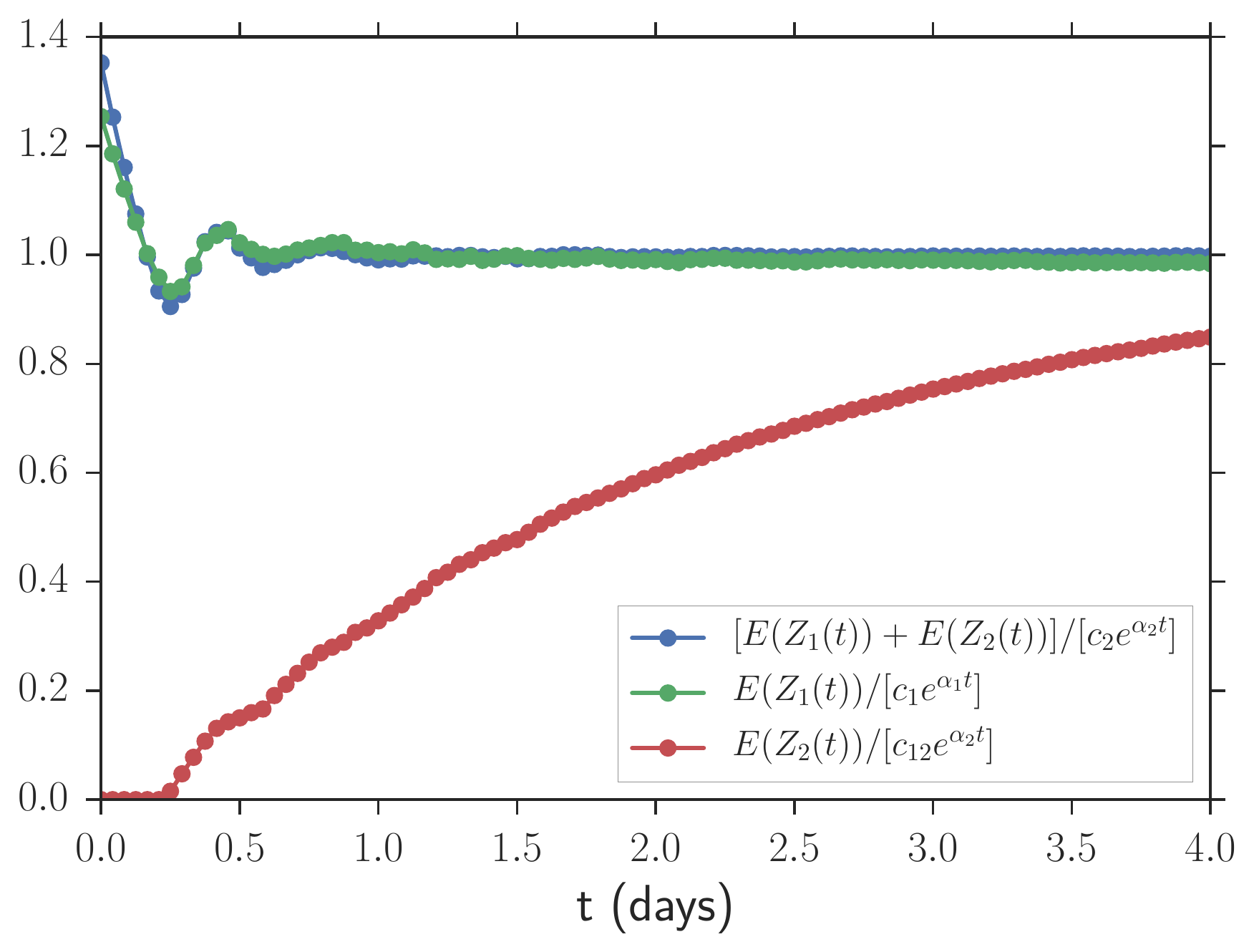}
}
\subfigure[normalised total generations ($\alpha_1<\alpha_2$)]{
\includegraphics[scale=0.4]{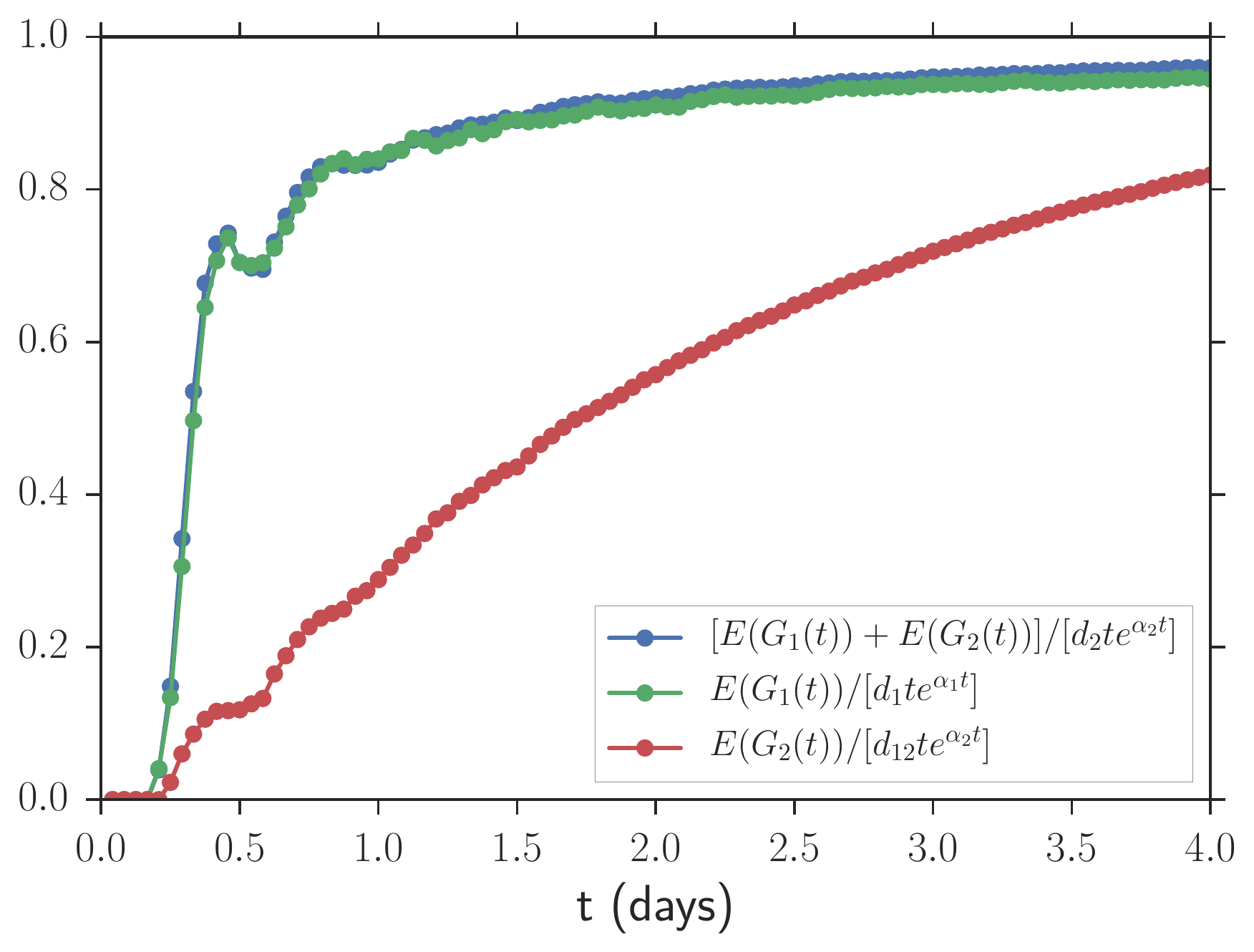}
}\\
\subfigure[normalised populations size ($\alpha_2<\alpha_1$)]{
\includegraphics[scale=0.4]{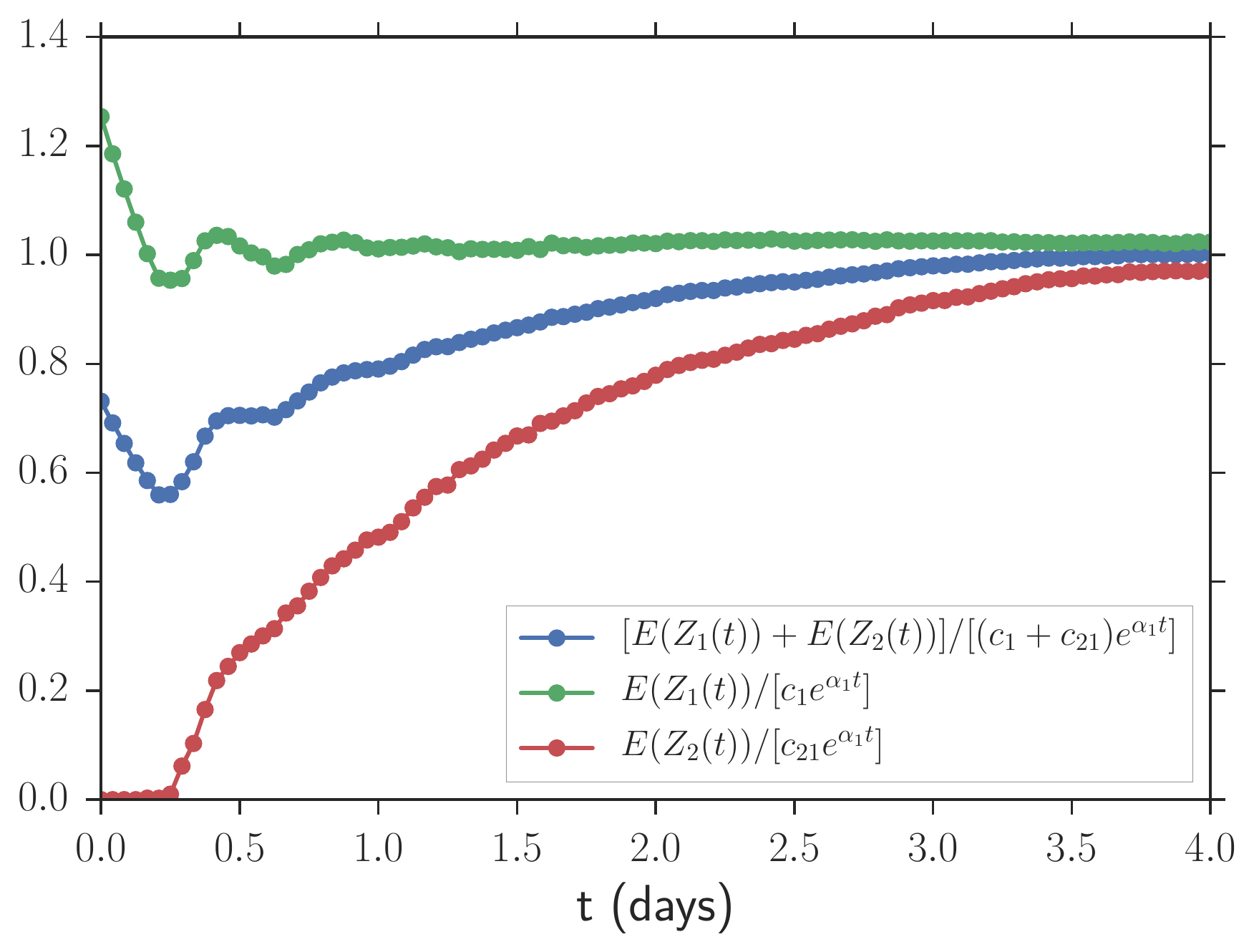}
}
\subfigure[normalised total generations ($\alpha_2<\alpha_1$)]{
\includegraphics[scale=0.4]{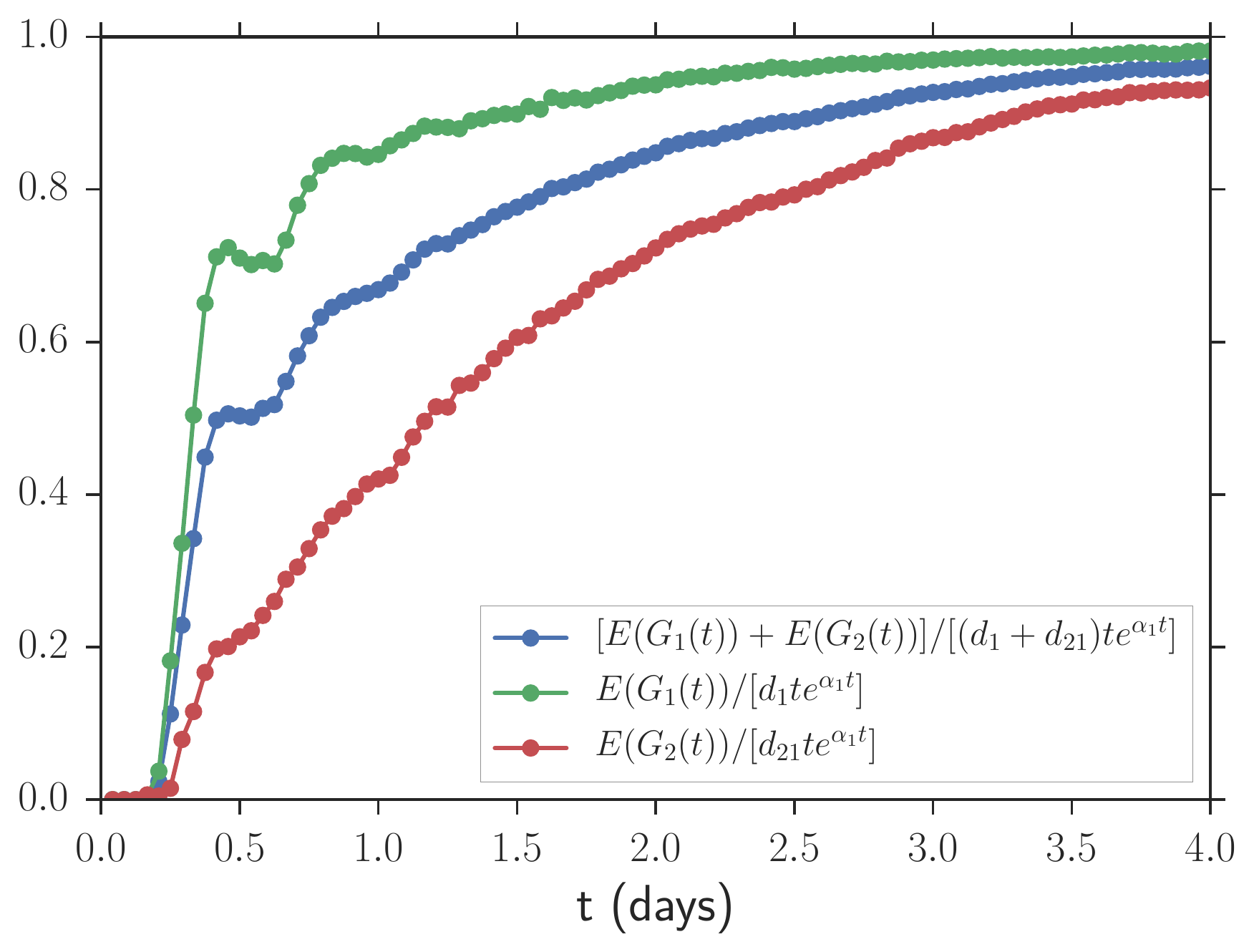}
}
\caption{{\bf Average growth rates of population sizes and total generations
of each type starting with a single type-1 cell and
using the scalings in Propositions~\ref{prop:multitype_alpha1<alpha2}
and~\ref{prop:multitype_alpha2<alpha1}}.  
Cells have lognormal
lifetime with mean 9.3 hours and standard deviation 2.54~\cite{Hawkins09}.
Type-1 cells give rise to type-1 cells with probability $5/6$ and
to type-2 cells with probability $1/6$. Means are computed averaging
the results of 1000 Monte Carlo simulations of populations growing
for four days.  (a)-(b) These illustrations are in the case
$\alpha_1<\alpha_2$ as both types of cells always have two offspring.
(c)-(d) These are in the setting $\alpha_2<\alpha_1$,
obtained by setting $N_1=2$ and $\P(N_2=0)=2/5=1-\P(N_2=2)$. 
}
\label{fig:multitype_Zi_Gi}
\end{figure}

We conclude the paper by presenting some simulated results that
illustrate the features of these two-type results, both for average
generation and for its inference.
Fig.~\ref{fig:multitype_Zi_Gi} provides average normalised paths
of the processes $Z_i(t)$ and $G_i(t)$. In
Fig.~\ref{fig:multitype_Zi_Gi}(a-b), $\alpha_1<\alpha_2$, but despite
the fact the type-2 population is the fastest growing on average,
it is the slowest one to converge. This occurs due to the random delay in
the production of any type-2 cells. Note also that the total
population of both type-1 and type-2 cells behave as a single-type
branching process with $N=2$ and log-normal lifetime 
distribution. Hence, the growth rates of $Z(t)=Z_1(t)+Z_2(t)$ and
$G(t)=G_1(t)+G_2(t)$ are the same as if the type-2 population was
started with one type-2 cell. In Fig.~\ref{fig:multitype_Zi_Gi}(c-d),
$\alpha_1>\alpha_2$. Here, the second population is dominated by
differentiation from the first cell type, with both populations have
the growth rate of the type-1 population. The behaviour of $Z(t)$
and $G(t)$ for the entire population is the sum of the
corresponding processes for the two types.

\begin{figure}
\centering
\subfigure[\label{fig:multitype_WZ_WG_a}]{
\includegraphics[scale=0.4]{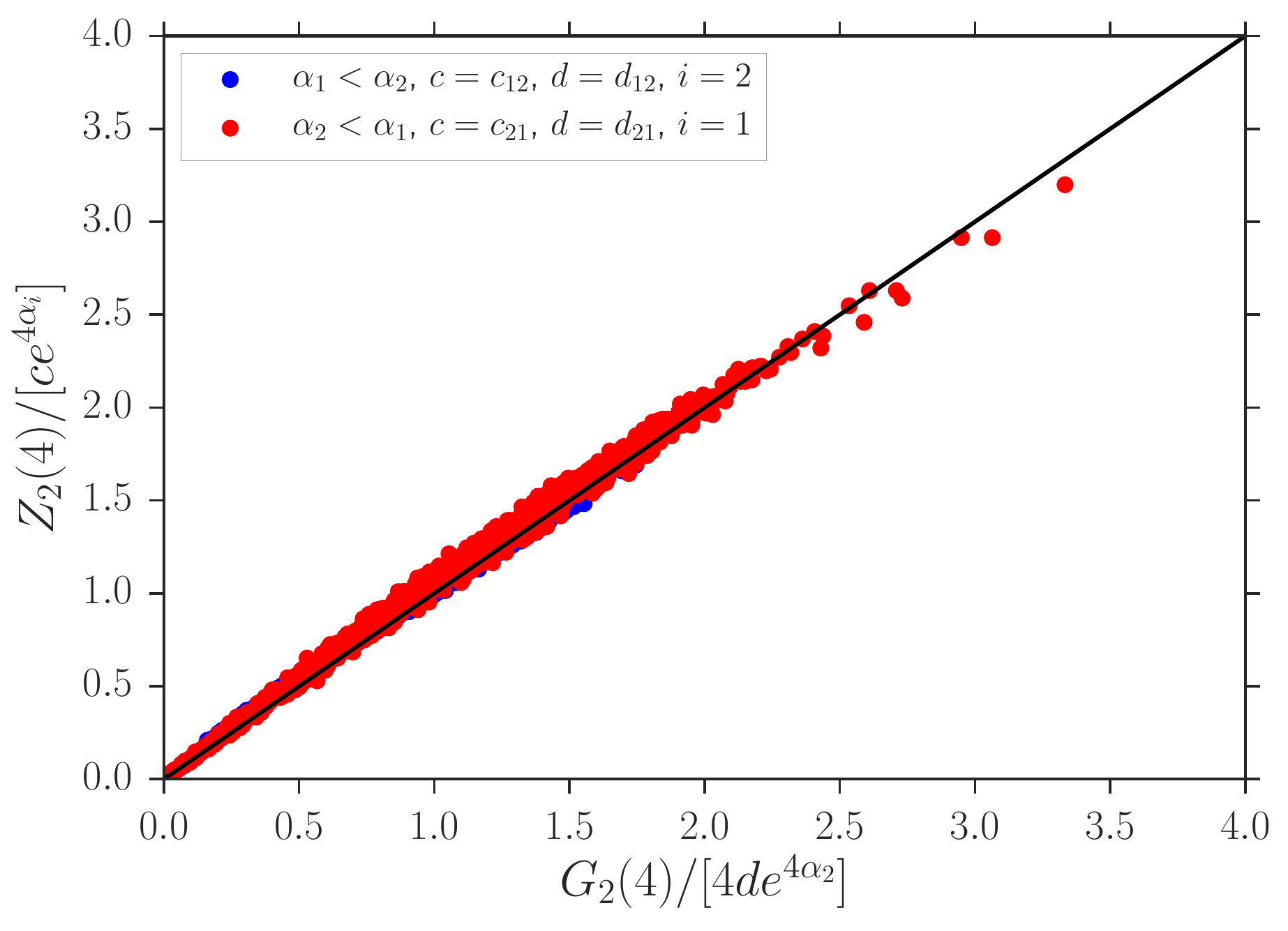}
}
\subfigure[\label{fig:multitype_WZ_WG_b}]{
\includegraphics[scale=0.4]{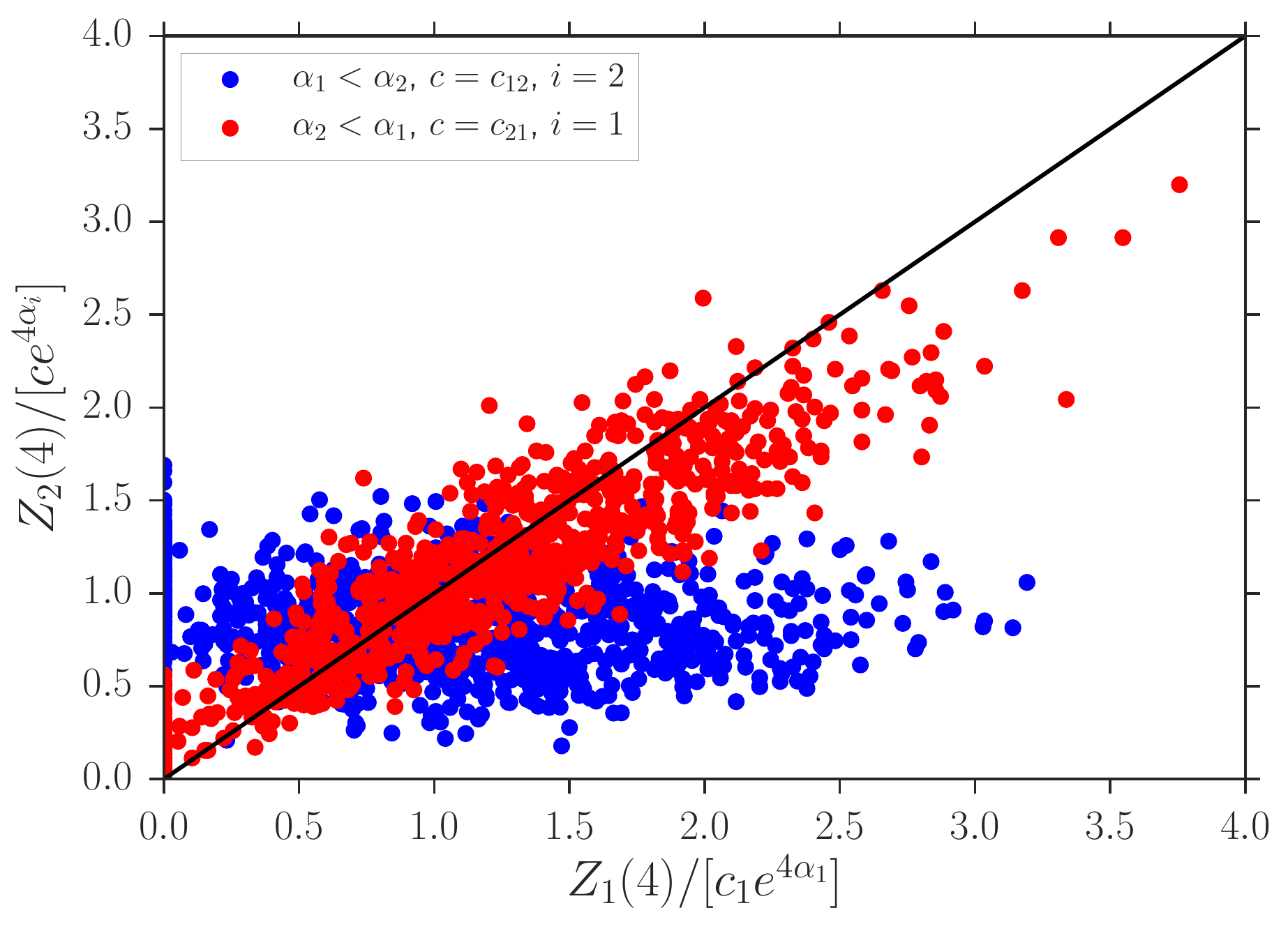}
}
\subfigure[\label{fig:multitype_WZ_WG_c}]{
\includegraphics[scale=0.4]{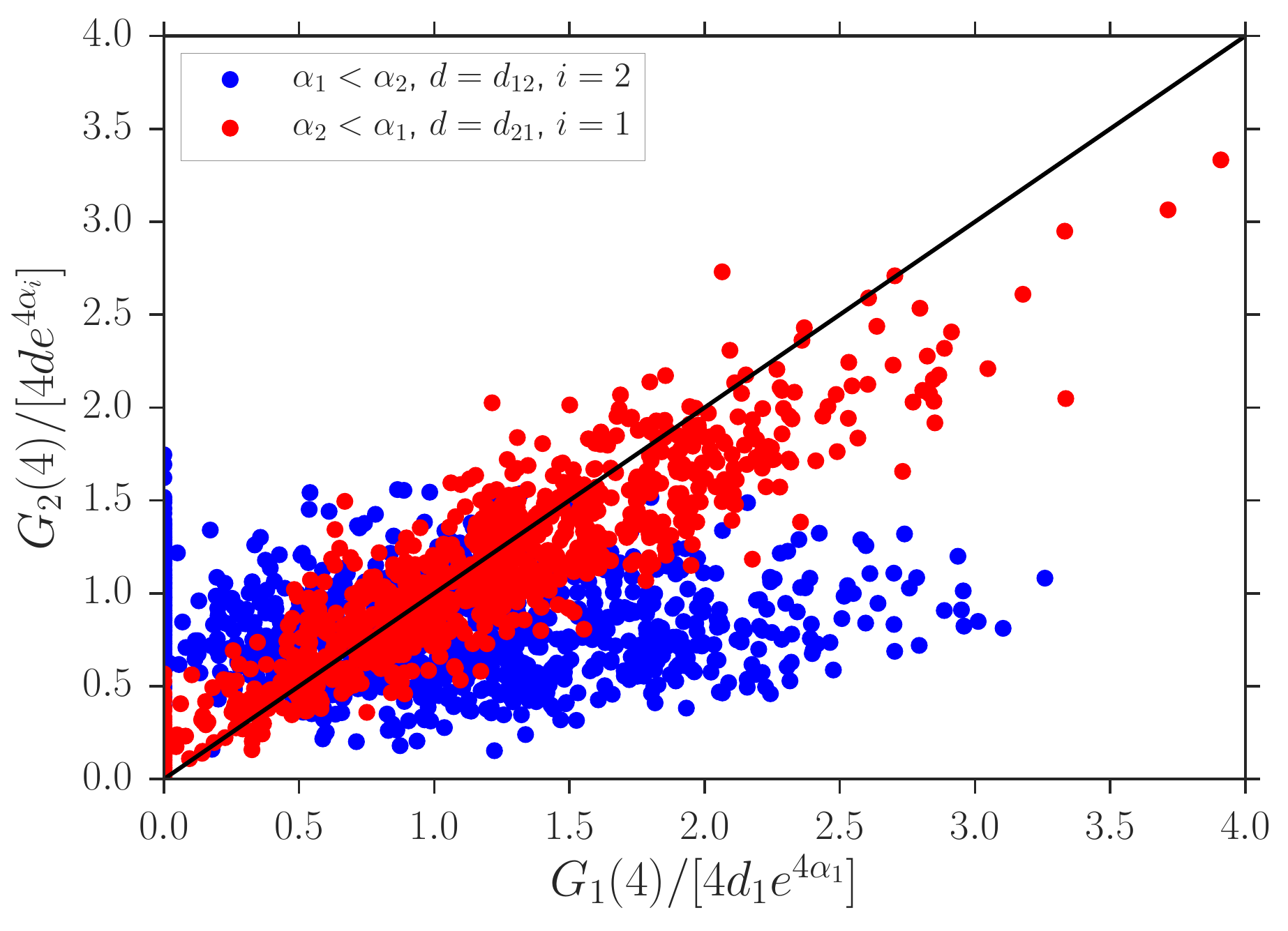}
}
\caption{{\bf Relationships in per-path randomness.}
Plots were created using
the same 1000 Monte Carlo simulations used to generate
Fig.~\ref{fig:multitype_Zi_Gi}. Blue points correspond to $\alpha_1<\alpha_2$, while red ones 
to $\alpha_2<\alpha_1$. 
(a) Scatter plot of normalised versions of $Z_2(t)$ and $G_2(t)$ is displayed at $t=4$
days. Pearson correlation coefficient for both blue and red points is $0.99$.
(b) Scatter plot of normalised versions of $Z_1(t)$ and $Z_2(t)$ is displayed at $t=4$
days. Pearson correlation coefficient for blue and red points is $-0.19$ and $0.94$, respectively.
(c) Scatter plot of normalised versions of $G_1(t)$ and $G_2(t)$ is displayed at $t=4$
days. Pearson correlation coefficient for blue and red points is $-0.09$ and $0.94$, respectively.
}
\label{fig:multitype_WZ_WG}
\end{figure}

Turning to the relatedness in random prefactors,
Fig.~\ref{fig:multitype_WZ_WG_a} is consistent with the deduction
that there is equality almost surely between the rescaled limit of
the population size and total generation of the second type.
Fig.~\ref{fig:multitype_WZ_WG_b} shows the prefactor for type-1
and type-2 population sizes. Consistent with results
in~\cite{jagers1969proportions}, red dots are suggestive that when
$\alpha_2<\alpha_1$ both normalised processes converge to the same
random variable. For $\alpha_1<\alpha_2$, however, this is not the
case for the blue dots and the random variables appear uncorrelated.
Fig.~\ref{fig:multitype_WZ_WG_c} is analogous to
Fig.~\ref{fig:multitype_WZ_WG_b} but for total generation, with the
same deduction as for the population size holding where when
$\alpha_1>\alpha_2$, the randomness is common to both types and
otherwise it is not.

\begin{figure}
\centering
\subfigure[\label{fig:multitype_estimator_paths_a}$\alpha_1<\alpha_2$, $Z_1(0)=1$]{
\includegraphics[scale=0.4]{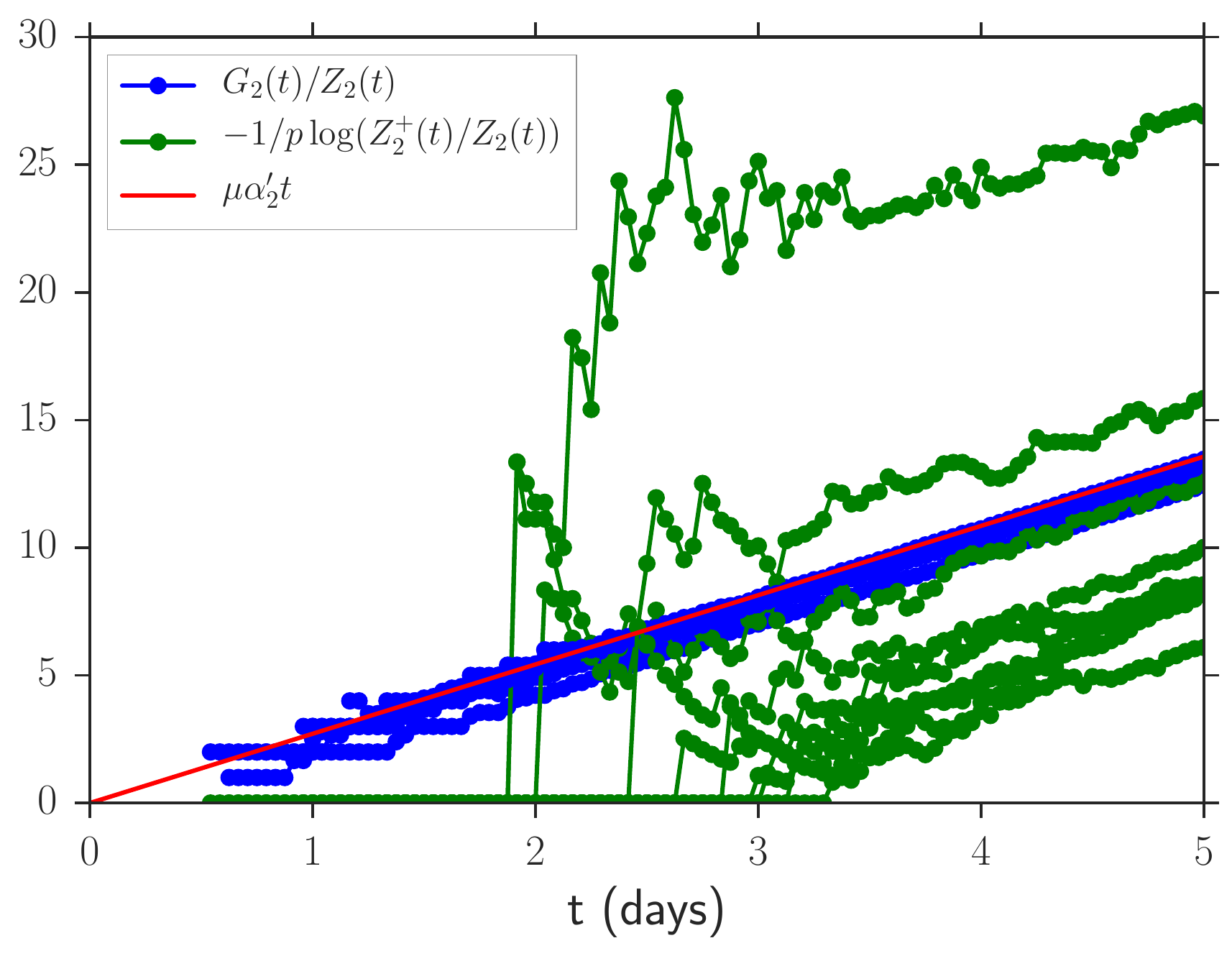}
}
\subfigure[\label{fig:multitype_estimator_paths_b}$\alpha_1<\alpha_2$, $Z_1(0)=100$]{
\includegraphics[scale=0.4]{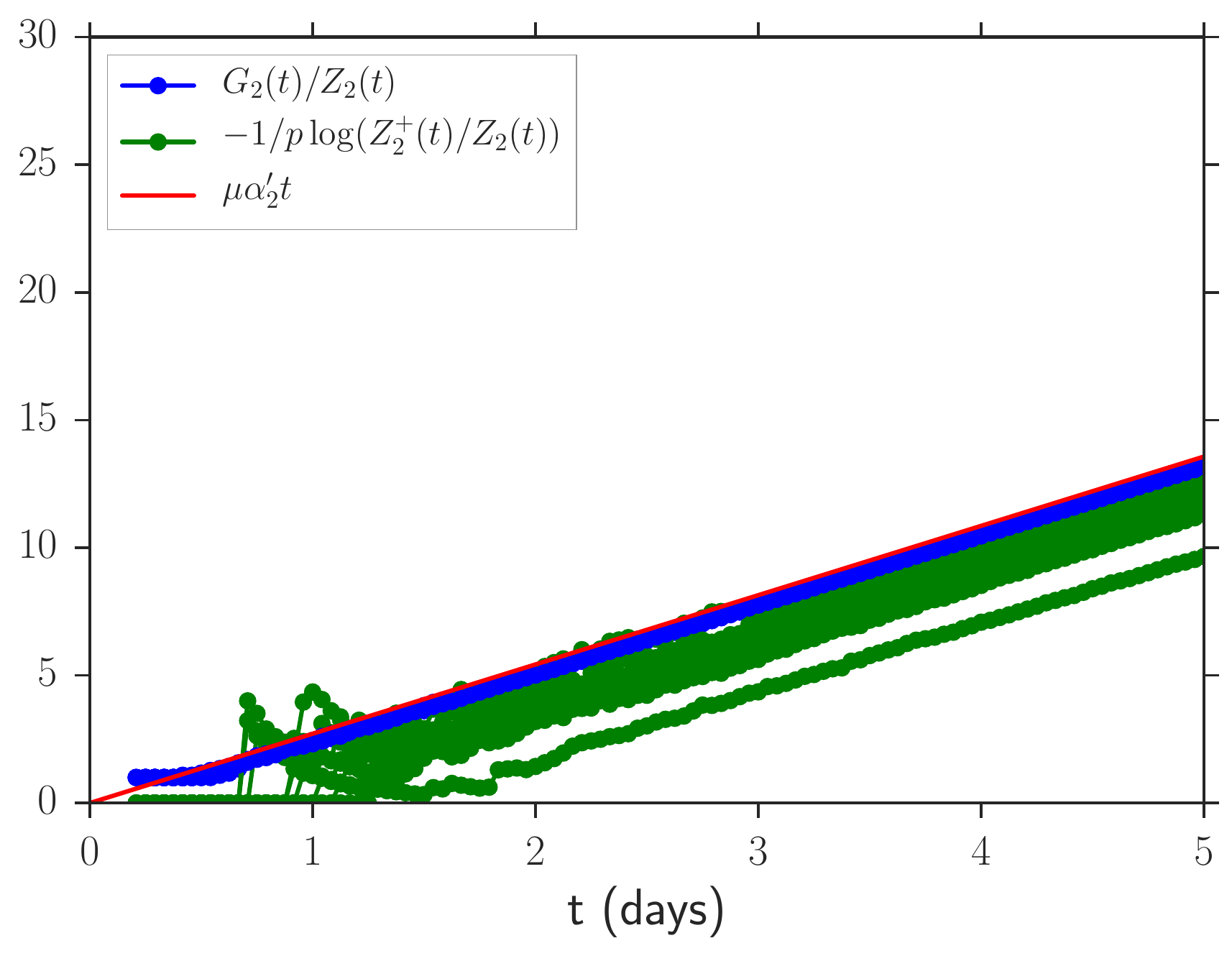}
}\\
\subfigure[\label{fig:multitype_estimator_paths_c} $\alpha_2<\alpha_1$, $Z_1(0)=1$]{
\includegraphics[scale=0.4]{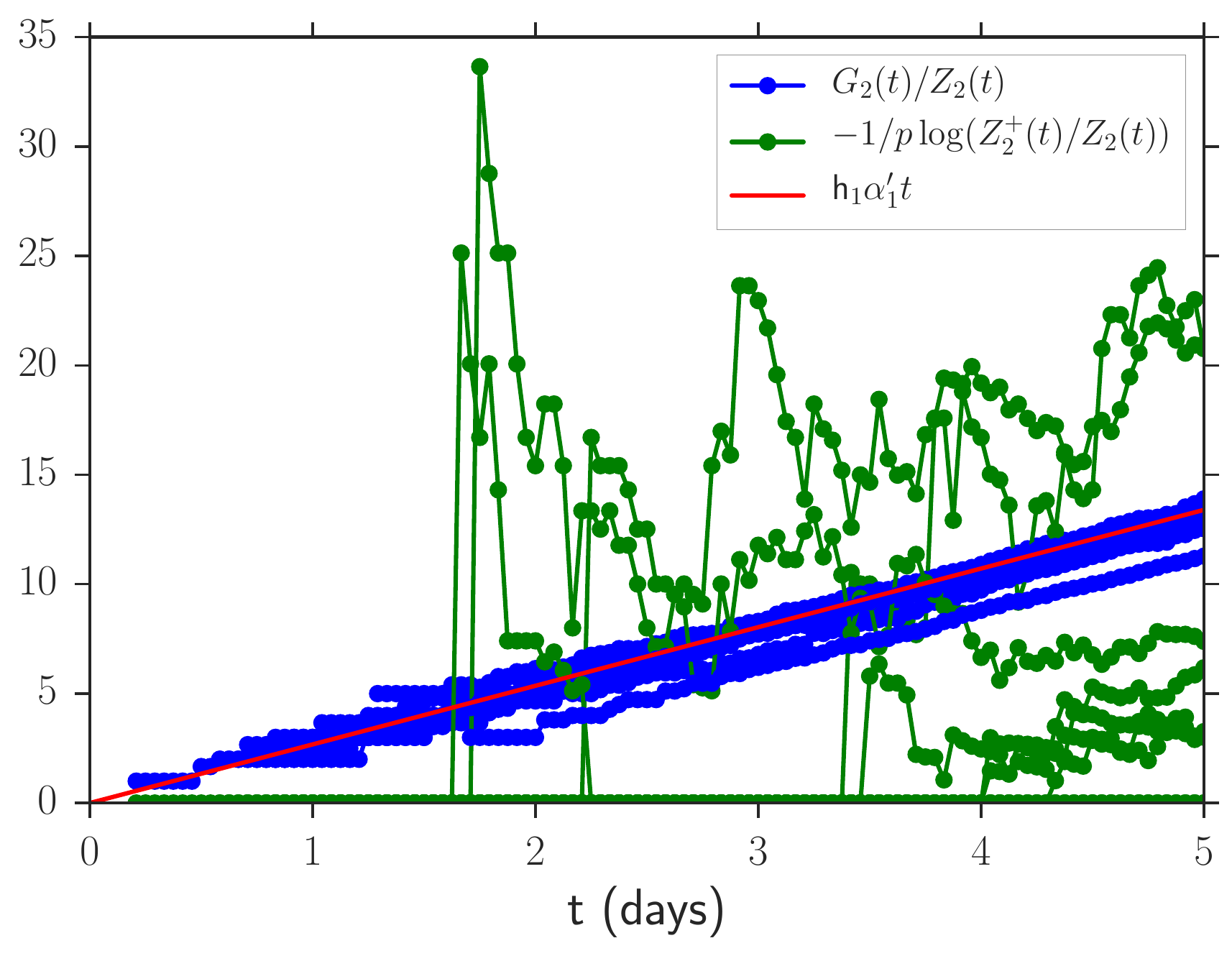}
}
\subfigure[\label{fig:multitype_estimator_paths_d}$\alpha_2<\alpha_1$, $Z_1(0)=100$]{
\includegraphics[scale=0.4]{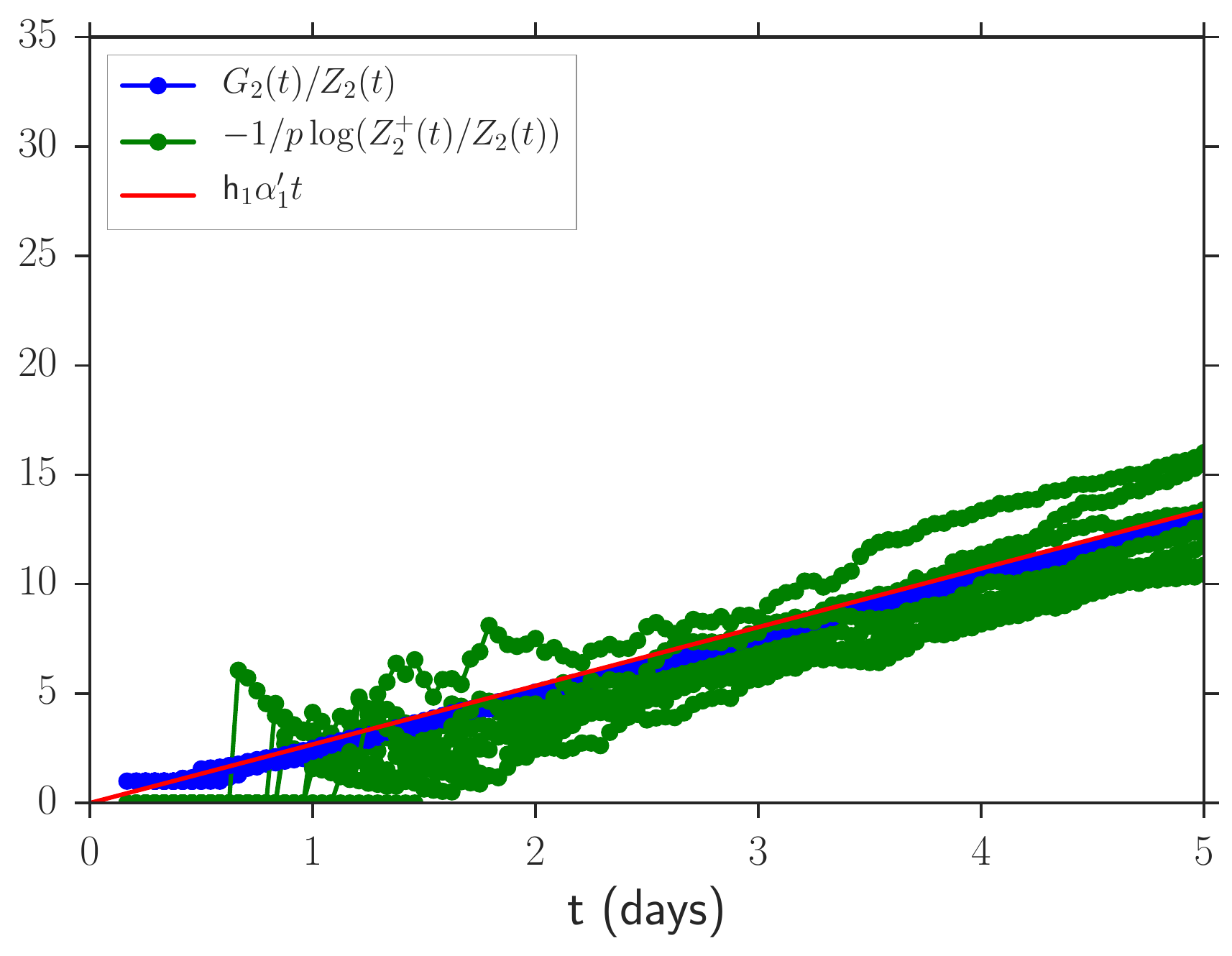}
}
\caption{{\bf Sample-path estimation of average generation.} 
For each sub-panel, ten Monte Carlo simulations of a two-type
population are presented. These employ the same parameterisation in
Fig.~\ref{fig:multitype_Zi_Gi}, with the exception of the initial
population size in the two right hand side panels. Each initial
cell is equipped with a neutral label that doesn't alter population
dynamics, and which is lost irrevocably to all subsequent offspring
with probability $p=10^{-2}$ per cell division. The red line indicates
the theoretical prediction of the mean average generation. Blue
lines indicate the development of the per-path average generation,
while the green lines are the estimates from the delabelling
formula~\eqref{eq:theformula}.
(a-b) Plots are in the setting $\alpha_1<\alpha_2$ case, but start
with one and 100 type-1 cells at $t=0$, respectively.
(c-d) Equivalent of (a-b) but with $\alpha_2<\alpha_1$. 
}\label{fig:multitype_estimator_paths}
\end{figure}

Part of the significance of Propositions~\ref{prop:multitype_alpha1<alpha2}
and~\ref{prop:multitype_alpha2<alpha1} is that they provide an
instrument by which one can infer the average generation of each
of the populations in a two-type Bellman-Harris branching
process, generalising the results in~\cite[Proposition 2]{weber2016inferring}. In
the presence of cells equipped with a neutral label that is heritably
lost with a fixed probability at each division, the average generation
and a function of the proportion of label-positive cells of each
type share the same dominant term. The mathematical results say
that the slope of the average generation and the slope of the
estimator are the same when the probabilistic regularity of a large
population takes hold. Figs~\ref{fig:multitype_estimator_paths_a}
and~\ref{fig:multitype_estimator_paths_c} illustrate this relationship
for the type-2 population via the use of some Monte Carlo simulations
in the presence of a single initial label positive cell of type-1.
In this setting the large population regularlity only takes hold at later times.
Starting with more than one initially labelled cell, illustrated
with $100$ in Figs~\ref{fig:multitype_estimator_paths_b}
and~\ref{fig:multitype_estimator_paths_d}, results in the desired
asymptotic equivalence occuring at a much earlier time. 
For true cellular systems, the cell numbers are likely to be significantly
larger again.

{\bf Acknowledgments:}
This work was supported by Science Foundation Ireland grant 12 IP 1263.

\end{document}